\documentclass[11pt,a4paper,reqno,oneside]{amsart}%

\usepackage{marginnote}
\usepackage{tikz}
\makeatletter

\usepackage{xcolor}
\usepackage{color,soul}
\usetikzlibrary{patterns}

\setulcolor{red}
\setul{}{1pt}

\usetikzlibrary{shapes,snakes}
\usepackage{environ}

\definecolor{mycolor}{RGB}{255,251,204}
\tikzstyle{mybox} = [draw=red, very thick,
rectangle, rounded corners, inner sep=10pt, inner ysep=20pt]
\tikzstyle{fancytitle} =[fill=white, text=red]

\NewEnviron{blc}[1]{\begin{tikzpicture}
	\node [mybox] (box){%
		\begin{minipage}{0.95\textwidth}
		\BODY
		\end{minipage}
	};
	\end{tikzpicture}
}

\usepackage{amsfonts}
\usepackage{amssymb}
\usepackage[utf8]{inputenc}
\usepackage{amsmath}

\usepackage{amsthm}
\usepackage{mathtools}
\usepackage{graphicx}%
\usepackage{enumerate}
\usepackage{xcolor}
\usepackage[pdftex,colorlinks]{hyperref}

\newtheorem{thm}{Theorem}[section]
\newtheorem{prop}[thm]{Proposition}
\newtheorem{lemma}[thm]{Lemma}
\newtheorem{cor}[thm]{Corollary}

\theoremstyle{remark}

\newtheorem*{assumption*}{\assumptionnumber}
\providecommand{\assumptionnumber}{}
\makeatletter
\newenvironment{assumption}[1]
{%
	\renewcommand{\assumptionnumber}{\textbf{A.#1}}%
	\begin{assumption*}%
		\protected@edef\@currentlabel{A.#1}%
	}
	{%
	\end{assumption*}
}
\makeatother

\theoremstyle{definition}

\theoremstyle{remark}
\newtheorem{rem}[thm]{Remark}

\numberwithin{equation}{section}

\newcommand{\R}{\mathbb{R}}				      

\newcommand{\lnum}{\let\veqno\leqno}
\newcommand{\dps}{\displaystyle}	
\newcommand{\dom}{(0,T)\times (0,1)}
\newcommand{\domw}{(0,T)\times \omega}

\newcommand{\intq}{\int_0^T\int_0^1}
\newcommand{\intw}{\int_0^T\int_{\omega}}
\newcommand{\I}[1]{\int_0^T\int_0^1e^{2s\varphi}\left((s\lambda)\sigma a#1_x^2+(s\lambda)^2\sigma^2#1^2\right)  }
\newcommand{\n}[2]{\|#1\|_{_{#2}}}
\newcommand{\rhat}{\hat{\rho}}
\newcommand{\rast}{\rho_\ast}

\newcommand{\Gb}[2]{\int_{T/2}^{T}\int_0^1e^{2sA}\left[s\lambda\zeta a \left(|#1_x|^2+|#2_x|^2\right)+(s\lambda)^2\zeta^2 \left(|#1|^2+|#2|^2\right) \right]  }

\newcommand{\G}[2]{\int_0^T\int_0^1e^{2sA}\left[s\lambda\zeta a \left(|#1_x|^2+|#2_x|^2\right)+(s\lambda)^2\zeta^2 \left(|#1|^2+|#2|^2\right) \right]  }

\newcommand{\Ga}[2]{\int_0^{T/2}\int_0^1e^{2sA}\left[s\lambda\zeta a \left(|#1_x|^2+|#2_x|^2\right)+(s\lambda)^2\zeta^2 \left(|#1|^2+|#2|^2\right) \right]  }
\newcommand{\E}{E}
\newcommand{\F}{F}

\begin{document}

	\title[Null controllability of coupled degenerate system  with nonlocal terms]{Local null controllability of coupled degenerate systems with nonlocal terms and one control force}
	
	\author[R. Demarque]{R. Demarque}
	\address[R. Demarque]{\newline Departamento de Ciências da Natureza,
		Universidade Federal Fluminense,
		Rio das Ostras, RJ, 28895-532, Brazil}
	\email{reginaldodr@id.uff.br}
	\author[J. Límaco]{J. Límaco}
	\address[J. Límaco]{\newline  Departamento de Matemática Aplicada,
		Universidade Federal Fluminense,
		Niterói, RJ, 24020-140, Brazil}
	\email{jlimaco@vm.uff.br}
	
	\author[L. Viana]{L. Viana}
	\address[L. Viana]{\newline  Departamento de Análise,
		Universidade Federal Fluminense,
		Niter\'{o}i, RJ, 24020-140, Brazil}
	\email{luizviana@id.uff.br}
	


	\begin{abstract}
		\noindent

		In this paper, we are concerned with the internal control of a class of one-dimensional nonlinear parabolic systems with nonlocal and weakly degenerate diffusion coefficients. Our main theorem establishes a local null  controllability result with only one internal  control for a system of two equations. The proof, based on the ideias developed by Fursikov and Imanivilov, is obtained from the global null controllability of the linearized system provided by  \textit{Lyusternik's Inverse Mapping Theorem}.  This work extends the results  previously treated by the authors for just one equation. For the system, the main issue is to obtain similar results with just one internal control, which requires a new Carleman  estimate with the local term just depending on one of the state function.	
		\end{abstract}
	
	%
	\maketitle

	\section{Introduction}\label{intro}

	In this paper we will establish a local null controllability result for the degenerate system, with nonlocal terms, given by
	\begin{equation}
	\begin{cases}
	u_t-\left(\mu_1 \left(x,\int_0^1u \   \right)u_x \right)_x+f_1(t,x,u,v) =h\chi_\omega, & (t,x)\in \dom,\\
	v_t-\left(\mu_2 \left(x,\int_0^1v \   \right)v_x \right)_x+f_2(t,x,u,v)=0, & (t,x) \in \dom,\\
	u(t,0)=u(t,1)=v(t,0)=v(t,1)=0, &  t\in (0,T),\\
	u(0,x)=u_0(x) \ \ \mbox{ and }\ \ v(0,x)=v_0(x), & x\in  (0,1),
	\end{cases} 
	\label{pb1} 
	\end{equation}
	where $T>0$ is given and  $(u_0,v_0)$ is the initial data. Moreover, $h$ is the control function, $(u,v)$ is the associated state and $\chi_{\omega}$ represent the characteristic function of     $\omega=(\alpha,\beta)\subset\subset (0,1)$. Regarding the functions $\mu_1 ,\mu_2, f_1$ and $f_2$, we make the following assumptions:

	\begin{assumption}{1} \label{hyp_a}  Let $\ell_1,\ell_2:\R\to \R$ be  $C^1$ functions with bounded derivative and suppose that $\ell_i(0)>0$, for each $i=1,2$. We also consider $a \in C([0,1])\cap C^1((0,1])$ satisfying $a(0)=0$, $a>0$ in $(0,1]$, $a'\geq 0$ and  \begin{equation}\label{prop_a}
		xa'(x)\leq Ka(x),\ \ \forall x\in [0,1] \mbox{ and some } K\in [0,1).
		\end{equation}
		In other words, the function $a$ behaves $x^{\alpha}$, with $\alpha \in (0,1)$. 
		
		Throughout this article, we will consider the functions $\mu_1,\mu_2:[0,1]\times \R\to \R$, given by 
		\[	\mu_1 (x,r)=\ell _{1} (r)a(x)\ \ \ \mbox{ and } \ \ \ \mu_2 (x,r)=\ell_2(r)a(x).\]
		And, for the sake of simplicity, we will also consider $\ell_i(0)=1$,   for each $i=1,2$.
	\end{assumption}
	
	\begin{assumption}{2}\label{hyp_f} For $i\in \{ 1,2\}$, we suppose that $f_i:[0,T]\times [0,1]\times \R^2\to \R$ is a $C^1$ function, with bounded derivatives, such that $f_i(t,x,0,0)=0$.  We also consider $b_{ij}(t,x):=\partial_{j+2}f_i(t,x,0,0)\in  L^\infty(\dom)$, for any $i,j\in \{1,2\}$. And, let us assume that there exists $\omega_1\subset\subset \omega$ such that
		\[\inf \{b_{21}(t,x);\ (t,x)\in [0,T]\times \omega_1 \}>0.\]
	\end{assumption}
	
	The main purpose of this work is to prove the local null controllability of ($\ref{pb1}$) by means of one control. Precisely,
	we will obtain $h\in L^2(\dom)$ such that the associated state $(u,v)=(u(t,x),v(t,x))$ of \eqref{pb1} satisfies
	\begin{equation*}
	u(T,x)\equiv v(T,x)\equiv 0 \mbox{ for any } x\in [0,1],
	\end{equation*}
	at least if $\|(u_0,v_0)\|_{H_a^1 \times H_a^1}$ is sufficiently small, where $H_a^1$ is an appropriate weighted Hilbert space which will be defined later in Section \ref{sec-pre} .
	

	The main difficulty comes from  the fact that the diffusion coefficients degenerate at $x=0$ and have nonlocal terms, namely
	\begin{equation*}
	\left(\mu_1 \left(\cdot ,\int_0^1u \   \right)u_x \right)_x \text{  and } \left(\mu_2 \left(\cdot,\int_0^1u \   \right)u_x \right)_x
	\end{equation*}
	satisfy assumption \eqref{hyp_a}.
	
	It is important to remark that semilinear nondegenerate problems have been extensively studied over the last decades, see \cite{fattorini1971,fernandez2012,fernandez2000cost,fursikov1996,lebeau1995} for example.  
	
	However, it seems to us that there is also a large interest in degenerate operators when the degeneracy occurs at the boundary of the space domain. For instance, in \cite{oleinik1999mathematical}, it was developed a study about the Prandtl system for stationary flows, in which the related boundary layer system was reduced to a quasilinear degenerate parabolic equation. Degenerate operators also appear in probabilistic models, see \cite{feller1952parabolic,feller1954diffusion}, and in climate science, see \cite{floridia2014approximate}. 
	
	In the context of degenerated  systems, controllability was studied in the case of two coupled equations in \cite{cannarsa2009controllability,hassi2011null,hassi2011carleman}. Recently, Ait Benhassi et al., in  \cite{benhassi2017algebraic},
	generalize the Kalman rank condition for the null controllability to $n$-coupled linear degenerate parabolic systems with  $m$-controls.

	On the other hand, as it was pointed out in \cite{fernandez2012}, nonlocal terms type can be found in several natural phenomena, such as in the reaction-diffusion systems, see \cite{chang2003some}, and in nonlinear vibration theory, see \cite{medeiros2002vibrations} for example.

	In \cite{alabau2006carleman},  it was obtained the null controllability for the semilinear equation
	\begin{align}\label{cf}
	u_t-(a(x)u_x)_x+f(t,x,u)=h(t,x)\chi_{\omega }, \text{ where }(t,x)\in (0,T)\times (0,1). 
	\end{align}
	
	Based on this work, in \cite{jrl2016}, we have considered (\ref{cf}), replacing the second-order term $(au_x)_x$ by a specific degenerate nonlocal operator. In that new context, we have achieved a local null controllability result. For systems of parabolic equations, the
	main issue is often to reduce the number of control functions acting on the system (see \cite{benhassi2017algebraic,carreno2013local,coron2009null,fadili2017null}, for example), besides that,  as it was pointed out in \cite{ammar2011recent}, the problem of controlling coupled parabolic equations has a very different behavior with respect to the scalar case, for instance, boundary controllability is not equivalent to distributed controllability, approximate controllability is not equivalent to null controllability, and 
	``\textit{the list of open
		problems is long and there is a lot of work to be done in order to fully understand
		this challenging subject}''\cite{ammar2011recent}.  In this direction, the current work may be seen as a natural continuation of \cite{jrl2016} and a first step in order to understand parabolic system with nonlocal and degenerate diffusion coefficients of the type $\left(\mu \left(\cdot ,\int_0^1u \   \right)u_x \right)_x $.
	
	Our main result is the following:
	
	\begin{thm}\label{main}
		Under the assumptions on $\mu_1,\mu_2,f_1$ and $f_2$, the nonlinear system \eqref{pb1} is locally null-controllable at any time $T>0$, i.e., there exists $r>0$ such that, whenever $u_0,v_0\in H_a^1$ and $\n{(u_0,v_0)}{H_a^1}\leq r$, there exists a control $h\in L^2(\domw)$ associated to a state $(u,v)$ satisfying
		\begin{equation*}
		u(T,x)=v(x,T)=0, \text{ for every } x\in [0,1].
		\end{equation*}
	\end{thm}
	
	The proof of Theorem \ref{main} will follow standard arguments (see for instance \cite{fernandez2012}, \cite{clark2013theoretical}, \cite{imanuvilov2001remarks}, \cite{fernandez2004local} and \cite{fernandez2006some}), based on Lyusternik's Inverse Mapping Theorem, which can be found in \cite{fursikov1996} and \cite{imanuvilov2003carleman}. To be more specific, we will see that the desired result is equivalent to find a solution to the equation
	\begin{equation}\label{eq}
	H(u,v,h)=(0,0,u_0,v_0),
	\end{equation}
	where $H:\E\to \F$ is a $C^1$ mapping between two appropriate Hilbert spaces, defined by
	\begin{equation*}
	H (u,v,h) = (H_{1} (u,v,h), H_{2} (u,v,h), u(\cdot ,0), v(\cdot ,0) ),
	\end{equation*}
	where 
	\[\displaystyle H_{1} (u,v,h) = u_t-\left(\mu_1 \left(x,\int_0^1u \   \right)u_x \right)_x+f_1(t,x,u,v) -h\chi_\omega\]
	and
	\[\displaystyle H_{2} (u,v,h) = v_t-\left(\mu_2 \left(x,\int_0^1v \   \right)v_x \right)_x+f_2(t,x,u,v) . \]
	
	In order to use Lysternik's Theorem, we need to prove that $H'(0)$ is onto. It is equivalent to prove  a global null controllability result  to  the  linearization of \eqref{pb1} (see the system (\ref{pb-lin}) below). This approach relies on a suitable Carleman estimate for the solutions of the adjoint problem associated to \eqref{pb-lin} (see Proposition \ref{prop3.2}).
	
	This paper is organized as follows: Section \ref{sec-pre} contains notations we use and a preliminary result. In Section \ref{sec-carl}, we prove a Carleman type inequality to solutions of \eqref{adj}, which also allows us to obtain an Observality inequality. Section \ref{sec-NCLS} is concerned with the global null controllability of \eqref{pb-lin} as well as two crucial additional estimates. Section \ref{sec-main} is devoted to the proof of Theorem \ref{main}. In Section \ref{sec6}, we present some comments and remarks. At the end, we  provide three appendices where we sketch the proof of wellposedness of problem {\eqref{pb1}} and two minors results we use throughout the paper.
	\section{Preliminaries}\label{sec-pre}

	\noindent
	
	In this section we will  state some notations and results which are necessary to prove Theorem \ref{main}. At first, we need to introduce some weighted spaces related to the function $a$,  namely
			\begin{align*}
		& \begin{multlined}[t][0.9\textwidth]
		H_a^1:= \{  u\in L^2(0,1);\ u\mbox{ is absolutely continuous in } (0,1],\\
		\sqrt{a}u_x\in L^2(0,1) \mbox{ and } u(1)=u(0)=0\},
		\end{multlined}\\
		& H_a^2:= \{  u\in H_a^1;\ au_x\in H^1(0,1) \},
		\end{align*}
		with  the norms defined by\\
		$\|u\|_{H_a^1}^2:=\|u\|_{L^2(0,1)}^2+\|\sqrt{a}u_x\|_{L^2(0,1)}^2$
		and $\|u\|_{H_a^2}^2:=\|u\|_{H_a^1}^2+\|(au_x)_x\|_{L^2(0,1)}^2$. 
		
		Alabau-Boussouira at al. in \cite{alabau2006carleman} introduced and studied some of the main properties of these spaces.
		
		As we pointed out in the Introduction, Lyusternik's Theorem requires the proof of a global null controllability result to a linear problem. Hence, we will present  a wellposedness result to this kind of problem that will be used later on the definition of our spaces in Section \ref{sec-main}. That result can be seen  in \cite{fadili2017null}.  Let us consider the linear system
		\begin{equation}
		\begin{cases}
		u_t-\left(a(x)u_x \right)_x+b_{11}(t,x)u+b_{12}(t,x)v =G_1, & (t,x) \mbox{ in }\dom,\\
		v_t-\left(a(x)v_x \right)_x+b_{21}(t,x)u+b_{22}(t,x)v =G_2, & (t,x) \mbox{ in }\dom,\\
		u(t,1)=u(t,0)=v(t,0)=v(t,1)=0, &  t \mbox{ in } (0,T),\\
		u(0,x)=u_0(x) \ \ \mbox{ and }\ \ v(0,x)=v_0(x), & x\mbox{ in } (0,1).
		\end{cases} 
		\label{pb-lin1} 
		\end{equation}

		\begin{prop}\label{prop-WP-lin}
			For all $G_1,G_2\in L^2(\dom)$ and $u_0,v_0\in L^2(0,1)$, there exists a unique weak solution $u,v\in C^0([0,T];L^2(0,1))\cap L^2(0,T;H_a^1)$ of \eqref{pb-lin1} and there exists a positive constant $C_T$ such that
			\begin{multline*}
			\sup_{t\in[0,T]}\left(\n{u(t)}{L^2(0,1)}^2+\n{v(t)}{L^2(0,1)}^2\right) +\int_0^T\left(\n{\sqrt{a}u_x}{L^2(0,1)}^2+ \n{\sqrt{a}v_x}{L^2(0,1)}^2\right)\\ 
			\leq C_T\left(\n{u_0}{L^2(0,1)}^2+\n{v_0}{L^2(0,1)}^2 +\n{G_1}{L^2(\dom)}+\n{G_2}{L^2(\dom)} \right).
			\end{multline*}	
			Moreover, if $u_0,v_0\in H_a^1$, then
			\[u,v\in H^1(0,T;L^2(0,1))\cap L^2(0,T;H_a^2)\cap C^0([0,T];H_a^1), \]
			and there exists a positive constant $C_T$ such that
			\begin{multline}\label{ineq1}
			\sup_{t\in[0,T]}\left(\n{u(t)}{H_a^1}^2+\n{v(t)}{H_a^1}^2\right)\\
			+\int_0^T\left(\n{u_t}{L^2(0,1)}^2+\n{v_t}{L^2(0,1)}^2+\n{(au_x)_x}{L^2(0,1)}^2+ \n{(av_x)_x}{L^2(0,1)}^2\right)\\
			\leq C_T\left(\n{u_0}{H_a^1}^2+\n{v_0}{H_a^1}^2 +\n{G_1}{L^2(\dom)}+\n{G_2}{L^2(\dom)} \right)
			\end{multline}
		\end{prop}

	Now, let us present a Carleman estimate that will play an important role in the next section. First, consider $\omega'=(\alpha',\beta')\subset\subset \omega$ and let $\psi:[0,1]\to \R$ be a $C^2 $  function such that 
	\begin{align}\label{functions1}
	\psi(x):=\begin{cases}
	\phantom{-}\int_0^x \frac{y}{a(y)}dy,\ x\in [0,\alpha')\\
	-\int_{\beta'}^x \frac{y}{a(y)}dy,\ x\in [\beta',1].
	\end{cases}
	\end{align}
	For $\lambda$ sufficiently large, define
	\begin{multline}\label{functions}
	\theta(t):=\frac{1}{[t(T-t)]^4}, \ \eta(x):=e^{\lambda(|\psi|_\infty+\psi)},\ \sigma(x,t):=\theta(t)\eta(x) \mbox{ and }\\
	\varphi(x,t):=\theta(t)(e^{\lambda(|\psi|_\infty+\psi)}-e^{3\lambda|\psi|_\infty}).
	\end{multline}
	
	Let us also  consider the linear system
	\begin{equation}\label{adj-jlr}
	\left\{\begin{array}{ll}
	-\xi_t-\left(a\left(x\right)\xi_x \right)_x+c(t,x)\xi=F(t,x), & (t,x)\in \dom, \\
	\xi(t,1)=\xi(t,0)=0, & t\in (0,T),\\
	\xi(T,x)=\xi_T(x), & x\in(0,1),
	\end{array}\right.   \end{equation}
	where $a$ satisfies assumption \ref{hyp_a}, $c\in L^\infty(\dom)$,  $F\in L^2(\dom)$ and $\xi_T\in L^2(0,1)$.
	
	The following Carleman estimate, proved in \cite{jrl2016}, holds for the solution to \eqref{adj-jlr}:
	
	\begin{prop}\label{car-jlr}
		There exist $C>0$ and $\lambda_0,s_0>0$ such that every solution  $\xi$ of (\ref{adj-jlr}) satisfies, for all $s\geq s_0$ and $\lambda\geq \lambda_0$, 
		\begin{equation}
		\intq e^{2s\varphi}\left((s\lambda)\sigma a\xi_x^2+(s\lambda)^2\sigma^2\xi^2 \right) 
		\leq C\left(\intq e^{2s\varphi}|F|^2\   +(\lambda s)^3\intw e^{2s\varphi}\sigma^3\xi^2\   \right),\label{carl_jlr}
		\end{equation} 
		where the constants $C, \lambda_0,s_0$  only depends on $\omega$, $a$, $\n{c}{L^\infty(\dom)}$ and $T$.
	\end{prop}
	
	For the sake of simplicity, we will  introduce the operator
	\begin{equation*}
	I(s,\lambda,\xi):=\I{\xi}.
	\end{equation*}

	Finally, let us present the following version of \textit{Lyusternik's Inverse Mapping Theorem} that can be found for instance in \cite{alekseev1987optimal,fursikov1996}.
\begin{thm}[Lyusternik]\label{lyusternik}
		Let $E$ and $F$ be two Banach spaces, $H: E\to F$ a $C^1$ mapping and put $\eta_0=H(0)$. If $H'(0):E\to F$ is onto, then there exist $r>0$ and $\tilde{H}:B_r(\eta_0)\subset F\to E$ such that
\[		H(\tilde{H}(\xi))=\xi,\ \forall \xi\in B_r(\eta_0),\]
		that is, $\tilde{H}$ is a right inverse of $H$. In addition, there exists $K>0$ such that
		\[\n{\tilde{H}(\xi)}{E}\leq K\n{\xi-\eta_0}{F}, \forall \xi\in B_r(\eta_0).  \]
	\end{thm}

	\section{Carleman and Observability inequalities} \label{sec-carl}
	
	In order to prove that map $H$ is onto, we have   to prove a global null controllability result  to  the  linearization of \eqref{pb1}, given by
	\begin{equation}
	\begin{cases}
	u_t-\left(a(x)u_x \right)_x+b_{11}(t,x)u+b_{12}(t,x)v =h\chi_\omega+g_1, & (t,x) \mbox{ in }\dom,\\
	v_t-\left(a(x)v_x \right)_x+b_{21}(t,x)u+b_{22}(t,x)v =g_2, & (t,x) \mbox{ in }\dom,\\
	u(t,1)=u(t,0)=v(t,0)=v(t,1)=0, &  t \mbox{ in } (0,T),\\
	u(0,x)=u_0(x) \ \ \mbox{ and }\ \ v(0,x)=v_0(x), & x\mbox{ in } (0,1),
	\end{cases} 
	\label{pb-lin} 
	\end{equation}
	where $g_1$, $g_2$ and $h$ belong to appropriate $L^2$-weighted spaces which we will specify later on.  To this purpose it is crucial to obtain an appropriate  Carleman estimate for solutions to
	\begin{equation}
	\begin{cases}
	-y_t-\left(a(x)y_x \right)_x+b_{11}(t,x)y+b_{21}(t,x)z =F_1, & (t,x) \mbox{ in }\dom,\\
	-z_t-\left(a(x)z_x \right)_x+b_{12}(t,x)y+b_{22}(t,x)z =F_2, & (t,x) \mbox{ in }\dom,\\
	y(t,0)=y(t,1)=z(t,0)=z(t,1)=0, &  t \mbox{ in } (0,T),\\
	y(T,x)=y_T(x),\ z(T,x)=z_T(x), & x\in (0,1),
	\end{cases} 
	\label{adj} 
	\end{equation}
	which is the adjoint problem of \eqref{pb-lin}. 
	
	\begin{prop}\label{carl1}There exist positive constants $C,\lambda_0$ and $s_0$ such that, for any $s\geq s_0$, $\lambda\geq \lambda_0$ and any $y_{_T},z_{_T}\in L^2(\dom)$, the corresponding solution $(y,z)$ to \eqref{adj} satisfies	
		\begin{multline}
		\intq e^{2s\varphi}\left((s\lambda)\sigma a(y_x^2+z_x^2)+(s\lambda)^2\sigma^2(y^2+z^2) \right) \\
		\leq C\left(\intq e^{2s\varphi}s^4\lambda^4\sigma^4(|F_1|^2+|F_2|^2)\   +\intw e^{2s\varphi}s^8\lambda^8\sigma^8y^2\   \right).\label{carl}
		\end{multline} 
	\end{prop}
	
	\begin{proof}
		Firstly, we rewrite the first equation in \eqref{adj} as
		\[-y_t-\left(a(x)y_x \right)_x+b_{11}(t,x)y=F_1-b_{21}z,  (t,x) \mbox{ in }\dom.\]
		So, we apply Proposition \ref{car-jlr} for the case in which $\xi=y$, $F=F_1-b_{21}z$, $c=b_{11}$ and $\omega=\omega_1$ to obtain 
		\begin{equation*}
		I(s,\lambda,y) 
		\leq  C\bigg(\intq e^{2s\varphi}|F_1|^2+\n{b_{21}}{_{L^\infty}}\intq e^{2s\varphi}|z|^2 +\int_0^T\int_{\omega_1} e^{2s\varphi}\lambda^3s^3\sigma^3|y|^2 \bigg).
		\end{equation*}
		
		Proceeding in the same way for the second equation, we get an analogous inequality
		\begin{equation*}
		I(s,\lambda,z) 
		\leq  C\bigg(\intq e^{2s\varphi}|F_2|^2+\n{b_{12}}{_{L^\infty}}\intq e^{2s\varphi}|y|^2 +\int_0^T\int_{\omega_1} e^{2s\varphi}\lambda^3s^3\sigma^3|z|^2 \bigg).
		\end{equation*}
		Now, we add this two inequalities and take  $s$ and $\lambda$ sufficiently large such that $I(s,\lambda,y)$ will  absorb the integral  depend on $|y|^2$, and $I(s,\lambda,z)$  the integral depend on $|z|^2$. This will give us the following inequality
		\begin{equation*}
		I(s,\lambda,y)+I(s,\lambda,z)\leq C\left(\intq e^{2s\varphi}(|F_1|^2+|F_2|^2)+\int_0^1\int_{\omega_1} e^{2s\varphi}\lambda^3s^3\sigma^3(|y|^2+|z|^2) \right).
		\end{equation*}
		
		Thus, in order to obtain \eqref{carl},  it is sufficient to show that there exists a small  $\varepsilon>0$  such that 
		\begin{multline*}
		\int_0^1\int_{\omega_1}e^{2s\varphi}\lambda^3s^3\sigma^3|z|^2
		\leq \varepsilon I(s,\lambda,z)\\ + C\left(\intq e^{2s\varphi}s^4\lambda^4\sigma^4(|F_1|^2+|F_2|^2)+\int_0^1\int_{\omega} e^{2s\varphi}\lambda^8s^8\sigma^8|y|^2 \right).
		\end{multline*}
		
		Let us take $\chi\in C_0^\infty(\omega)$ satisfying $0\leq \chi\leq 1$ and $\chi\equiv 1$ in $\omega_1$. Since $\dps\inf b_{21}>0$, we can easily see that
		\[\int_0^1\int_{\omega_1}e^{2s\varphi}\lambda^3s^3\sigma^3|z|^2\leq C\intw \chi b_{21}e^{2s\varphi}\lambda^3s^3\sigma^3|z|^2.\]

		Now, multiplying the first equation in \eqref{adj} by $e^{2s\varphi}s^3\lambda^3\sigma^3\chi z$ and integrating over $\dom$, we get
		\begin{align*}
		\intq \chi b_{21} e^{2s\varphi}s^3\lambda^3\sigma^3 |z|^2  = &\intq  \chi e^{2s\varphi}s^3\lambda^3\sigma^3 zF_1+\intq \chi e^{2s\varphi}s^3\lambda^3\sigma^3zy_t
		\\
		& +\intq \chi e^{2s\varphi}s^3\lambda^3\sigma^3(ay_x)_xz-\intq \chi b_{11} e^{2s\varphi}s^3\lambda^3\sigma^3 yz\\
		=: & I_1+I_2+I_3+I_4.
		\end{align*}

		Next, we need to estimate $I_1,I_2,I_3$ and $I_4$. Firstly, from Young's inequality, we have
		\begin{align}\label{I1}
		I_1\leq &  C\intq \chi  e^{2s\varphi}s^3\lambda^3\sigma^3 |z||F_1|
		\leq  \varepsilon \intq e^{2s\varphi}s^2\lambda^2\sigma^2 |z|^2+C_\varepsilon\intq   \chi^2e^{2s\varphi}s^4\lambda^4\sigma^4 |F_1|^2\nonumber\\
		\leq & \varepsilon I(s,\lambda,z)+C_\varepsilon\intq e^{2s\varphi}s^4\lambda^4\sigma^4 |F_1|^2.
		\end{align}
		
		In the same way, since $b_{11}$ is bounded, it is immediate that
		\begin{align}\label{I4}
		I_4\leq &  C\intq \chi  e^{2s\varphi}s^3\lambda^3\sigma^3 |y||z|
		\leq  \varepsilon \intq e^{2s\varphi}s^2\lambda^2\sigma^2 |z|^2+C_\varepsilon\intq   \chi^2e^{2s\varphi}s^4\lambda^4\sigma^4 |y|^2\nonumber\\
		\leq & \varepsilon I(s,\lambda,z)+C_\varepsilon\intw e^{2s\varphi}s^8\lambda^8\sigma^8 |y|^2.
		\end{align}
		
		Using integration by parts,  we will split up $I_2$ and $I_3$ in several integrals. In fact, 
		\begin{multline*}
		I_3= -\intq \chi s^3\lambda^3\sigma^3 e^{2s\varphi} az_xy_x+\intq\left(a(\chi  e^{2s\varphi}\sigma^3)_x\right)_xs^3\lambda^3yz\\+\intq s^3\lambda^3(\chi  e^{2s\varphi}\sigma^3)_xaz_xy.
		\end{multline*}
		and, recalling  that $e^{2s\varphi} $ vanishes at $0$ and $T$ and  using the second equation of \eqref{adj}, we have
		\begin{align}
		I_2= & -\intq \left[\chi s^3\lambda^3(e^{2s\varphi}\sigma^3)_{_t}+\chi s^3\lambda^3\sigma^3e^{2s\varphi}b_{22}\right]yz +\intq\chi s^3\lambda^3\sigma^3e^{2s\varphi}(az_x)_xy\nonumber\\ 
		&\ \ \ \ -\intq\chi s^3\sigma^3e^{2s\varphi}b_{12}y^2 +\intq\chi s^3\lambda^3\sigma^3e^{2s\varphi}yF_2\nonumber\\ 
		=& -\intq \left[\chi s^3\lambda^3(e^{2s\varphi}\sigma^3)_{_t}+\chi s^3\lambda^3\sigma^3e^{2s\varphi}b_{22}\right]yz  -\intq s^3\lambda^3(\chi e^{2s\varphi}\sigma^3)_xaz_xy\nonumber\\ 
		& -\intq \chi s^3\lambda^3\sigma^3e^{2s\varphi}az_xy_x -\intq\chi s^3\lambda^3\sigma^3e^{2s\varphi}b_{12}y^2 +\intq\chi s^3\lambda^3\sigma^3e^{2s\varphi}yF_2. \label{I2}
		\end{align}
		
		Thus,
		\begin{align}\label{I2I3}
		I_2+I_3= &  \intq \left[-\chi s^3\lambda^3(e^{2s\varphi}\sigma^3)_{_t}-\chi s^3\lambda^3\sigma^3e^{2s\varphi}b_{22}+s^3\lambda^3\left(a(\chi  e^{2s\varphi}\sigma^3)_x\right)_x\right]yz \nonumber\\
		& -2\intq \chi s^3\lambda^3\sigma^3e^{2s\varphi}az_xy_x -\intq\chi s^3\lambda^3\sigma^3e^{2s\varphi}b_{12}y^2\nonumber\\
		& +\intq\chi s^3\lambda^3\sigma^3e^{2s\varphi}yF_2 \nonumber \\
		=& J_1+J_2+J_3+J_4.
		\end{align}
		
		Now, it remains estimates these four integrals. It is immediate that
		\begin{equation}\label{J3}
		J_3\leq C\intw e^{2s\varphi} s^8\lambda^8\sigma^8y^2.
		\end{equation}
		and, from Young's inequality, that
		\[J_4\leq \intq e^{2s\varphi}s^4\lambda^4\sigma^4|F_2|^2+\intw e^{2s\varphi}s^8\lambda^8\sigma^8 y^2.\]

		In order to estimate $J_1$, we will analyze each term between brackets. Firstly, we observe that all the terms are  multiplied by $\chi$, which vanishes outside of $\omega$. Clearly, 
		\[|\chi s^3\lambda^3\sigma^3e^{2s\varphi}b_{22}|\leq C\chi s^5\lambda^5 \sigma^5 e^{2s\varphi}. \]
		Since $|\sigma_x|\leq C\lambda \sigma$, $|\sigma_{xx}|\leq C\lambda^2 \sigma$  and $a\in C^1(\omega)$, after distributing the  derivatives with respect to $x$, we can see that 
		\[|s^3\lambda^3\left(a(\chi  e^{2s\varphi}\sigma^3)_x\right)_x|\leq  C\chi s^5\lambda^5 \sigma^5 e^{2s\varphi}.\]
		Likewise, the relations $|\varphi_t|\leq C\sigma^2$ and  $|\sigma_t|\leq C\sigma^2$ yield
		\[|-\chi s^3\lambda^3(e^{2s\varphi}\sigma^3)_{_t}|\leq  C\chi s^5\lambda^5 \sigma^5 e^{2s\varphi}.\]

		As a conclusion,
		\begin{equation*}
		J_1\leq C\intq \chi s^5\lambda^5 \sigma^5 e^{2s\varphi}|yz|\leq \varepsilon I(s,\lambda,z)+C_\varepsilon \intw s^8\lambda^8\sigma^8e^{2s\varphi}y^2.
		\end{equation*}

		The last step is to deal with $J_2$. To do this, we notice that 
		\begin{align*}
		J_2\leq & \varepsilon\intq e^{2s\varphi}s\lambda\sigma az_x^2+C_\varepsilon\intq \chi^5e^{2s\varphi}s^5\lambda^5\sigma^5y_x^2\\
		\leq & \varepsilon I(s,\lambda,z)+C_\varepsilon\intq \chi^5e^{2s\varphi}s^5\lambda^5\sigma^5y_x^2.
		\end{align*}
		Hence, we only need to estimate the last integral. Multiplying the first equation in \eqref{adj} by $\chi^2e^{2s\varphi}s^5\lambda^5\sigma^5y$, integrate over $\dom$ and integrating by parts we get that
		\begin{align*}
		&    \intq \chi^2e^{2s\varphi}s^5\lambda^5\sigma^5y_x^2 \\
		&\phantom{\leq } \leq  \intq \left[-\frac{1}{2}\chi^2s^5\lambda^5(e^{2s\varphi}\sigma^5)_{_t}+ \frac{1}{2}s^5\lambda^5\left[a(e^{2s\varphi}\sigma^5\chi^2)_{_x}\right]_x-\chi^2b_{11}e^{2s\varphi}s^5\lambda^5\sigma^5\right]y^2\\
		&\phantom{\leq }\phantom{\leq } -\intq \chi^2b_{21}e^{2s\varphi}s^5\lambda^5\sigma^5yz+\intq \chi^2e^{2s\varphi}s^5\lambda^5\sigma^5yF_1.
		\end{align*}
		
		We can see that all the integrals here are of the same type of those in \eqref{I2I3}. Following the same arguments developed there, we have the result.
	\end{proof}
	
	Now we need to prove a Carleman inequality  for solutions of problem \eqref{adj}  with weights which do not vanish at $t=0$. It is necessary in order to guarantee  the null controllability results in Theorem {\ref{linearcontrol}} and Theorem {\ref{main}}. We will give more details in Remark {\ref{rem1}}. 
	
	Consider a function $m\in C^\infty([0,T])$ satisfying
	
	\[\left\{\begin{array}{ll}
	m(t)\geq t^4(T-t)^4, & t\in (0, T/2];\\
	m(t)= t^4(T-t)^4, & t\in \left[T/2,T\right];\\
	m(0)>  0, &
	\end{array}
	\right.\]
	and define
	
	\[\tau(t):=\frac{1}{m(t)},\ \ \zeta(x,t):=\tau(t)\eta(x)\ \ \mbox{ and }\ \ A(t,x):=\tau(t)(e^{\lambda(|\psi|_\infty+\psi)}-e^{3\lambda|\psi|_\infty}),\]
	where $(t,x)\in [0,T)\times [0,1]$. As usual, we introduce the operators
	\begin{align*}
	&\Gamma(s,\xi,\vartheta):=\G{\xi}{\vartheta}, \\
	&\Gamma_1(s,\xi,\vartheta):=\Ga{\xi}{\vartheta}\\
	\text{and }& \\
	&  \Gamma_2(s,\xi,\vartheta):=\Gb{\xi}{\vartheta}.
	\end{align*}

	\begin{prop}\label{prop3.2}
		There exist positive constants $C,\lambda_0$ and $s_0$ such that, for any $s\geq s_0$, $\lambda\geq \lambda_0$ and any $y_{_T},z_{_T}\in L^2(\dom)$ the corresponding solution $(y,z)$ to \eqref{adj} satisfies	
		\begin{multline}
		\G{y}{z}\\
		\leq C\left(\intq e^{2sA}s^4\lambda^4\zeta^4(|F_1|^2+|F_2|^2)\   +\intw e^{2sA}s^8\lambda^8\zeta^8|y|^2\   \right).\label{carlA}
		\end{multline} 
	\end{prop}
	
	\begin{proof} In order to estimate $\Gamma_2(s,y,z)$, let us   observe that   $e^{2s\varphi}\sigma^n\leq Ce^{2sA}\zeta^n$ for all $(t,x)\in [0,T]\times[0,1]$ and $n\geq 0$. Since $\tau=\theta$ and $A=\varphi$ in $[T/2,T]$, Carleman inequality \eqref{carl} implies
		\[\Gamma_2(s,y,z)\leq  C\left(\intq e^{2sA}s^4\lambda^4\zeta^4(|F_1|^2+|F_2|^2)\   +\intw e^{2sA}s^8\lambda^8\zeta^8|y|^2\   \right).\]
		
		Now, we will prove an analogous estimate for $\Gamma_1(s,y,z)$, arguing as in \cite{clark2013theoretical}. Multiplying the first and the second equations of \eqref{adj} by $-y$ and $-z$, respectively, and integrating over $[0,1]$, we obtain
		\begin{multline}\label{eq3.5}
		-\frac{d}{dt}\left(\n{y}{L^2(0,1)}^2+\n{z}{L^2(0,1)}^2\right)-C \left(\n{y}{L^2(0,1)}^2+\n{z}{L^2(0,1)}^2\right)\\+2\left(\n{\sqrt{a}y_x}{L^2(0,1)}^2+\n{\sqrt{a}z_x}{L^2(0,1)}^2\right)
		\leq \n{F_1}{L^2(0,1)}^2+\n{F_2}{L^2(0,1)}^2,
		\end{multline}
		which implies
		\[-\frac{d}{dt}\left[e^{Ct}\left(\n{y}{L^2(0,1)}^2+\n{z}{L^2(0,1)}^2\right)\right]\leq e^{Ct}\left(\n{F_1}{L^2(0,1)}^2+\n{F_2}{L^2(0,1)}^2\right).\]
		
		Integrating from a $t\in [0,T/2]$ to $t+T/4$, we get
		\begin{multline*}
		\n{y}{L^2(0,1)}^2+\n{z}{L^2(0,1)}^2
		\leq   e^{CT}\int_0^{3T/4}\left( \n{F_1}{L^2(0,1)}^2+\n{F_2}{L^2(0,1)}^2\right)\\
		+e^{3CT/4}\left( \n{y(t+T/4)}{L^2(0,1)}^2+\n{z(t+T/4)}{L^2(0,1)}^2\right).
		\end{multline*}
		
		Hence, we conclude that
		\begin{multline}\label{eq3.6}
		\int_{0}^{T/2}	\left(\n{y}{L^2(0,1)}^2+\n{z}{L^2(0,1)}^2\right)\leq  e^{CT}\frac{ T}{2}\int_0^{3T/4}\left( \n{F_1}{L^2(0,1)}^2+\n{F_2}{L^2(0,1)}^2\right)\\
		+e^{3CT/4}\int_{T/4}^{3T/4}\left( \n{y}{L^2(0,1)}^2+\n{z}{L^2(0,1)}^2\right).
		\end{multline}

		Now, integrating inequality \eqref{eq3.5}  over $[0,t]$, where $t\in [0,T]$, we take
		\begin{multline}\label{eq3.6-2}
		\int_0^t\left(\n{\sqrt{a}y_x}{L^2(0,1)}^2+\n{\sqrt{a}z_x}{L^2(0,1)}^2\right)\leq \frac{1}{2}\left(\n{y}{L^2(0,1)}^2+\n{z}{L^2(0,1)}^2\right)
		\\+C\left[ \int_0^t\left(\n{y}{L^2(0,1)}^2+\n{z}{L^2(0,1)}^2\right) +\int_0^t\left(\n{F_1}{L^2(0,1)}^2+\n{F_2}{L^2(0,1)}^2\right)\right].
		\end{multline}		
		
		In order to establish our next inequality, first we recall that if a function $f$ is non-negative, then the function $t\mapsto \int_{t_0}^t f(t) $ is non-decreasing. As a consequence, for all $  t\in [T/2,3T/4]$, we have that
		\begin{align*}
		\int_0^{T/2}\left(\n{\sqrt{a}y_x}{L^2(0,1)}^2+\n{\sqrt{a}z_x}{L^2(0,1)}^2\right) & \leq \int_{0}^{t}\left(\n{\sqrt{a}y_x}{L^2(0,1)}^2+\n{\sqrt{a}z_x}{L^2(0,1)}^2\right)\\
		& 	\leq \int_{0}^{3T/4}\left(\n{\sqrt{a}y_x}{L^2(0,1)}^2+\n{\sqrt{a}z_x}{L^2(0,1)}^2\right).	\end{align*}
		Thus, integrating inequality \eqref{eq3.6-2} from $T/2$ to $3T/4$ and  using \eqref{eq3.6} we have
		\begin{multline}\label{eq3.7}
		\int_0^{3T/4}\left(\n{\sqrt{a}y_x}{L^2(0,1)}^2+\n{\sqrt{a}z_x}{L^2(0,1)}^2\right)\\ \leq C\left[\int_{T/2}^{3T/4}\left(\n{y}{L^2(0,1)}^2+\n{z}{L^2(0,1)}^2\right)+\int_0^{3T/4}\left(\n{F_1}{L^2(0,1)}^2+\n{F_2}{L^2(0,1)}^2\right)\right].
		\end{multline}
		
		Finally, we observe that $e^{2sA}(s\lambda \zeta)^n$ and $e^{2s\varphi}(s\lambda \sigma)^n$ are bounded in $[0,T/2]$ and $[T/4,3T/4]$ \\ respectively, for all $n\in \mathbb{Z}$. Hence, \eqref{eq3.6}, \eqref{eq3.7} and Carleman Inequality \eqref{carl} imply
		\begin{align*}
		\Gamma_1(s,y,z)& =\Ga{y}{z} \\
		& \leq C\left(\int_{T/4}^{3T/4}\int_{0}^{1}\left(|y|^2+|z|^2\right)+\int_{0}^{3T/4} \int_{0}^{1}\left(|F_1|^2+|F_2|^2\right)\right)\\
		& \leq C\left(\int_{T/4}^{3T/4}\int_{0}^{1}e^{2s\varphi}(s\lambda)^2\sigma^2\left(|y|^2+|z|^2\right) +\int_{0}^{3T/4}\int_{0}^{1}e^{2sA}(s\lambda)^4\sigma^4\left(|F_1|^2+|F_2|^2\right)\right)\\
		& 	\leq C\bigg(\intq e^{2sA}s^4\lambda^4\zeta^4(|F_1|^2+|F_2|^2)\   +\intw e^{2sA}s^8\lambda^8\zeta^8|y|^2\bigg),		
		\end{align*} 
		which concludes the proof.
	\end{proof}

	\begin{cor}\label{DO-2}
		There exist positive constants $C,\lambda_0$ and $s_0$ such that, for any $s\geq s_0$, $\lambda\geq \lambda_0$ and any  $y_{_T},z_{_T}\in L^2(0,1)$ the corresponding solution $(y,z)$ to \eqref{adj},  with $F_1\equiv F_2\equiv0$,  satisfies
		\begin{equation}\label{obs.ineq}
		\n{y(0)}{L^2(0,1)}^2+\n{z(0)}{L^2(0,1)}^2\leq C\intw e^{2sA}s^8\lambda^8\zeta^8|y|^2
		\end{equation}
	\end{cor}
	
	\begin{proof}
		Using  standard energy inequalities for each equations in \eqref{adj}, we obtain
		\[-\frac{d}{2dt}\left(\n{y}{L^2(0,1)}^2+\n{z}{L^2(0,1)}^2\right)\leq 2C \left(\n{y}{L^2(0,1)}^2+\n{z}{L^2(0,1)}^2\right).\]
		Hence, we get that
		\begin{equation}\label{eq1}
		\n{y(0)}{L^2(0,1)}^2+\n{z(0)}{L^2(0,1)}^2\leq e^{4CT}\left(\n{y}{L^2(0,1)}^2+\n{z}{L^2(0,1)}^2\right).
		\end{equation}
		Finally, integrating the last inequality in $[0,3T/4]$,  recalling that $e^{2sA}(s\lambda)^2\zeta^2$ is bounded from below in $[0,3T/4]$ and using \eqref{carlA} with $F_1\equiv F_2\equiv 0$,  we obtain
		\begin{equation*}
		\n{y(0)}{L^2(0,1)}^2+\n{z(0)}{L^2(0,1)}^2 \leq C\int_{0}^{\frac{3T}{4}}\int_{0}^{1}e^{2sA}(s\lambda)^2 \zeta^2(|y|^2+|z|^2) \leq  C\intw e^{2sA} s^8\lambda^8\zeta^8|y|^2.
		\end{equation*}

	\end{proof}
	
	\section{Global null controllability for the linear system} \label{sec-NCLS}
	
	\noindent
	
	The goal of this section is to prove a  null controllability result for the linear system \eqref{pb-lin} and establish some important additional estimates.  In order to state this result, we need to define the weights functions:
	\[\rho=e^{-sA},\ \rho_0=e^{-sA}\zeta^{-1},\ \rhat=e^{-sA}\zeta^{-5/2} \text{ and } \rast=e^{-sA}\zeta^{-4}, \] 
	which satisfy $\rast\leq C\rhat\leq C\rho_0\leq C\rho$ and $\rhat^2=\rast\rho_0$.

	\begin{thm} \label{linearcontrol}
		If $u_0,v_0\in H_a^1(0,1)$ and the functions $g_1$ and $g_2$ fulfill 
		\[\intq \rho_0^2\left(|g_1|^2+|g_2|^2\right)<+\infty,\]
		then the system \eqref{pb-lin} is null-controllable. More precisely, there exists a control $h\in L^2(\domw)$ with associated state $(u,v)$ satisfying
		\begin{equation}\label{eq25}
		\intw\rast^2|h|^2<+\infty \text{ and } \intq \rho_0^2(|u|^2+|v|^2)<+\infty.
		\end{equation}
		In particular, $u(T,x)=v(T,x)=0,$ for all $x\in [0,1]$.
	\end{thm}
	
	\begin{rem}\label{rem1}
		\noindent
		
		\begin{itemize}
			\item[(a)] 
				Recalling that  $\rho_0(t)\to +\infty$, as $t\to T^{-}$, and $\rho_{0} (0) >0$ (since $m(0)>0$), the second relation in ($\ref{eq25}$) garantees $u(T,x)=v(T,x)=0$. 
			\item[(b)]  If we had chosen $m\in C^{\infty} ([0,T])$ satisfying $m(0)=0$, we would verify $\rho_0(t)\to +\infty$, as $t\to 0^{+}$. As a consequence, the second relation in ($\ref{eq25}$) would imply $u(0,x)=v(0,x)\equiv 0$. However, in general, this fact is not true, because $u_{0} \in H_{a}^{1} (0,1)$ and $v_{0} \in H_{a}^{1} (0,1)$ must be taken arbitrarily.  
		\end{itemize}
		
	\end{rem}
	
	\begin{proof}
	For  each $n\in \mathbb{N}^\ast$, we define
		\[A_n(t,x)=\frac{A(T-t)^4}{(T-t)^4+\frac{1}{n}},\ (t,x) \in  [0,T]\times [0,1].
		\]
		We also consider
		\[\rho_n=e^{-sA_n}, \ \rho_{0,n}=\rho_{n}\zeta^{-2} \text{ and } \rho_{\ast,n}=\rho_n\zeta^{-4} m_n, \text{ where }m_n=\begin{cases}
		1,&  x\in \omega,\\
		n,&  x\notin \omega.
		\end{cases}\]
		
		These weight functions are built in such a way that $\rho_{0,n}$ and $\rho_{\ast, n}$ are  bounded from below by a constant depending only on $T$ and from above by another one depending on $n$ and $T$, see Lemma \ref{lemma-const}. It will allows us to obtain a sequence $(u_n,v_n,h_n)_{n\in \mathbb{N}^\ast}$ of solutions to \eqref{pb-lin} which will converge to a solution of \eqref{pb-lin} satisfying \eqref{eq25}. 
		
		To do that,  for any functions $u,v,h\in L^2(\dom)$, let us  define the functional
		\[J_n(u,v,h)=\frac{1}{2}\intq \rho_{0,n}^2\left(|u|^2+|v|^2\right)+\frac{1}{2}\intq \rho_{\ast,n}^2|h|^2.\]
		
		Since each $J_n$ is lower semi-continuous, strictly convex and coercive (see Appendix \ref{Jn}), Proposition 1.2 in \cite{ekeland1999convex} yields  a unique $(u_n,v_n,h_n)$ that minimizes $J_n(u,v,h)$ subject to the condition $\mathcal{C}=\{(u,v,h)\in \mathbf{ \left[L^2(\dom)\right]^3}; (u,v,h)\text{ solves } \eqref{pb-lin}\}$. In this case, $(u_n,v_n,h_n)$ satisfies \eqref{pb-lin} and, by virtue of \textit{Lagrange's Principle}, there exist  functions $p_n,q_n$ solving the following system
		\begin{equation}\label{S2}
		\begin{cases}
		-{p_{n}}_t-({ap_{n}}_x)_x+b_{11}p_{n}+b_{21}q_{n}=-\rho_{0,n}^2u_n, & (t,x) \in \dom,\\
		-{q_{n}}_t-(a{q_n}_x)_x+b_{12}p_{n}+b_{22}q_{{n}}=-\rho_{0,n}^2v_n, & (t,x)\in \dom,\\
		p_n(t,0)=p_n(t,1)=q_n(t,0)=q_n(t,1)=0,& t\in (0,T),\\
		p_n(T,x)=q_n(T,x)=0, & x\in (0,1),\\
		p_n=\rho_{_{\ast,n}}^2h_n \chi_{\omega }, & \in \dom.
		\end{cases}
		\end{equation}
	
		We wan to prove that $(J_n(u_n,v_n,h_n))_{n=1}^\infty$ is a numerical bounded sequence. To do that, we will use Carleman  and   Observability inequalities to prove that  $J_{n} (u_{n},v_n ,h_{n} ) \leq C \sqrt{J_{n} (u_{n},v_n ,h_{n} )}$.
		
		In fact, multiplying the PDEs in (\ref{S2}) by $u_{n}$ and $v_n$, integrating over $\dom$ and  using integration by parts, we have
		\begin{align*}
		0&=
		\begin{multlined}[t][0.9\textwidth]
		\intq [-{p_{n}}_t - (a {p_{n}}_x )_{x} +b_{11} p_{n} +b_{21}q_n +\rho_{0,n}^{2} u_{n}] u_{n} \\
		+	\intq [-{q_{n}}_t - (a {q_{n}}_x )_{x} +b_{12} p_{n} +b_{22}q_n  +\rho_{0,n}^{2} v_{n}] v_{n}
		\end{multlined}\\
		&=\begin{multlined}[t][0.9\textwidth]
		\intq [{u_{n}}_t - (a {u_{n}}_x )_{x} +b_{11}u_{n}+b_{12}v_n ]p_{n} + \intq \rho_{0,n}^{2} |u_{n} |^{2} + \int_{0}^{1} p_{n} (0,x) u_{0} \\
		+\intq [{v_{n}}_t - (a {v_{n}}_x )_{x} +b_{21}u_{n}+b_{22}v_n ]q_{n} + \intq \rho_{0,n}^{2} |v_{n} |^{2} + \int_{0}^{1} q_{n} (0,x) v_{0} 
		\end{multlined}
		\end{align*}
		
		\begin{align*}
		&=\begin{multlined}[t][0.9\textwidth]
		\intq (h_{n} \chi _{\omega} +g_1) p_{n} + \intq \rho_{0,n}^{2} |u_{n} |^{2} + \int_{0}^{1} p_{n} (0,x) u_{0}\\
		+\intq g_2 q_{n} + \intq \rho_{0,n}^{2} |v_{n} |^{2} + \int_{0}^{1} q_{n} (0,x) v_{0}
		\end{multlined}
		\end{align*}
		Hence, since $p_n=p_{\ast,n}^2h_n\chi_{\omega }$, we obtain
		\begin{align*}
		\displaystyle J_{n} (u_{n}, v_n ,h_{n}) & =\frac{1}{2}\intq \rho_{0,n}^2\left(|u_n|^2+|v_n|^2\right)+\frac{1}{2}\intq \rho_{\ast,n}^2|h_n|^2\chi_{\omega } \\
		&\begin{multlined}[t][0.8\textwidth]
		=-\frac{1}{2} \intq (h_{n} \chi _{\omega} +g_1) p_{n} -\frac{1}{2} \int_{0}^{1} p_{n} (0,x) u_{0} \\
		-\frac{1}{2} \intq g_2q_{n} - \frac{1}{2} \int_{0}^{1} q_{n} (0,x)v_{0} + \frac{1}{2} \intq p_{n} h_{n}\chi_\omega 
		\end{multlined}\\
		& =-\frac{1}{2} \intq (g_1 p_{n} + g_2q_{n}) -\frac{1}{2} \int_{0}^{1} p_{n} (0,x) u_{0} - \frac{1}{2} \int_{0}^{1} q_{n} (0,x)v_{0}.
		\end{align*}
		Using  H\"older inequality, we have
		\begin{multline*}
		\displaystyle J_{n} (u_{n},v_n ,h_{n})		 \leq \frac{1}{2} \n{g_1\rho _{0}}{ L^2}\n{\rho _{0}^{-1} p_{n}}{L^2}+\frac{1}{2} \n{g_2\rho _{0}}{ L^2}\n{\rho _{0}^{-1} q_{n}}{L^2}\\
		+\frac{1}{2}\n{p_{n} (0,\cdot)}{L^2(0,1)}\n{u_{0}}{L^2(0,1)} +\frac{1}{2}\n{q_{n} (0,\cdot)}{L^2(0,1)}\n{v_{0}}{L^2(0,1)}.
		\end{multline*}

		Applying the classical Cauchy-Schwartz inequality $\Sigma_{i=1}^4(a_ib_i)\leq (\Sigma_{i=1}^4a_i^2)^{1/2}(\Sigma_{i=1}^4b_i^2)^{1/2}$ and then using hypothesis $\displaystyle \intq \rho _{0}^{2}( |g_1|^{2}+|g_2|^{2}) <+\infty$ , we get
		\begin{align*}
		\displaystyle J_{n} (u_{n},v_n ,h_{n})	 &\begin{multlined}[t][0.9\textwidth]
		\leq \frac{1}{2} \left(\n{g_1\rho _{0}}{ L^2}^2+\n{g_2\rho _{0}}{ L^2}^2 +\|u_{0}\|^2 +\|v_{0}\|^2\right)^{1/2}\\
		\cdot\left(\n{\rho _{0}^{-1} p_{n}}{L^2}^2 +\n{\rho _{0}^{-1} q_{n}}{L^2}^2 + \n{p_{n} (0,\cdot)}{L^2(0,1)}^2+ \n{q_{n} (0,\cdot)}{L^2(0,1)}^2 \right)^{1/2}
		\end{multlined} \\
		&\leq C \left( \intq \rho _{0}^{-2} (p_{n}^{2}+ q_{n}^{2})  +\n{p_{n} (0,\cdot)}{L^2(0,1)}^2+ \n{q_{n} (0,\cdot)}{L^2(0,1)}^2 \right) ^{1/2}.
		\end{align*}
		
		Now, it is enough to prove that each term in the last inequality is  bounded by $J(u_n,v_n,h_n)$. In order to estimate the first term, we will apply Carleman inequality \eqref{carlA}  to the solution $(p_n,q_n)$ to \eqref{S2} and then use that  $\rho^{-2}\rho_n^4 \leq \rho_n^{2}$  and  $p_n=\rho_{\ast,n}^2h_n\chi_{\omega }$ in $\dom$. 
		
		Indeed, 
		\begin{align*}
		\displaystyle \intq \rho _{0}^{-2} (p_{n}^{2}+q_n^2)
		&=\frac{1}{(s\lambda)^2} \intq e^{2s A}(s\lambda)^2\zeta^2(p_{n}^{2}+q_n^2)\\
		&\leq C\left( \intq e^{2s A}\zeta^4(s\lambda)^4\rho_{0,n}^{4}(  |u_{n} |^{2}+|v_n|^2) + \intw  e^{2s  A}\zeta^8(s\lambda)^8 |p_{n} |^{2} \right) \\
		&\leq C \left( \intq \rho^{-2} \rho_{n}^{4}\zeta^{-4}( |u_{n} |^{2}+|v_{n} |^{2}) + \intw \rho^{-2}\zeta^8 \rho _{\ast ,n}^{4} h_{n}^{2} \right) \\
		&\leq C \left( \intq \rho_{0,n}^{2} ( |u_{n} |^{2}+|v_{n} |^{2}) + \intw \rho _{\ast ,n}^{2} |h_{n} |^{2} \right) \\
		&=C J_{n} (u_{n}, v_n ,h_{n}). 
		\end{align*} 
		
		The remain terms are readily estimated  from the Observability Inequality \eqref{obs.ineq}. In fact,
		\begin{align*}
		\n{p_{n} (0,\cdot)}{L^2(0,1)}^2+ \n{q_{n} (0,\cdot)}{L^2(0,1)}^2 & \leq C\left( \intw  e^{2s  A}\zeta^8(s\lambda)^8 |p_{n} |^{2} \right)\leq C J_{n} (u_{n}, v_n ,h_{n}).
		\end{align*}
		
		Hence, we have proven that $J_{n} (u_{n},v_n ,h_{n} ) \leq C \sqrt{J_{n} (u_{n},v_n ,h_{n} )}$.  As a consequence, $(J_n(u_n,v_n,h_n))_{n\in \mathbb{N}}$ is a bounded sequence. Since $\rho_{0,n}^2\geq C_T$ and $\rho_{\ast,n}^2\geq C_Tm_n$, we deduce that
		\[\n{u_n}{L^2}^2+\n{v_n}{L^2}^2+\intw |h_n|^2 +n\int_0^T\int_{[0,1]\setminus\omega}|h_n|^2\leq C J_{n} (u_{n},v_n ,h_{n} )\leq  C.
		\]

		It implies that   there exist $u,v\in L^2(\dom)$ and $h\in L^2(\domw)$  such that
		\[u_n\rightharpoonup u, \ \ v_n\rightharpoonup v \ \mbox{ in } L^2(\dom) \ \mbox{ and } \ \ h_n\chi_{\omega }\rightharpoonup h\ \ \mbox{ in } L^2(\domw),\] 
		up to subsequences.
		From this, we take
		\begin{equation}\label{wc}
		\rho_{0,n}u_n\rightharpoonup \rho_0 u \ \ \rho_{0,n}v_n\rightharpoonup \rho_0 v \ \ \mbox{ and } \ \ \rho_{\ast,n}h_n\chi_{\omega }\rightharpoonup \rho_{\ast}h\chi_{\omega } \ \ \mbox{ in } L^2(\dom).
		\end{equation} 
		
		Consequently, passing to limits as $n\to +\infty$, we conclude that $(u,v,h)$ solves \eqref{pb-lin}.		Furthermore, (\ref{eq25}) follows from  (\ref{wc}) and this establishes the result. 
			\end{proof}
	
	%
	The next step is to prove two crucial estimates which will needed later.
	
	\begin{prop}\label{ad1}
		Assume the same hypothesis of Theorem $\ref{linearcontrol}$. Then
		
		\begin{multline}
		\intq \rhat ^2 a(|u_x |^2 + |v_x |^2 ) 
		\leq C \left( \intq \rho_{0}^{2} (|u|^2 + |v|^2 ) + \intw  \rho_{\ast}^{2} |h|^2 \right) \\   
		+C\left( \intq \rho _{0}^{2} (|g_1 |^2 + |g_2 |^2 ) + \| u_0 \| _{H_{a}^{1}}^{2} + \| v_0 \| _{H_{a}^{1}}^{2} \right)  \ .
		\end{multline} 
	\end{prop} 
	
	\begin{proof}
		Let us multiply the first equation in (\ref{pb-lin}) by $\rhat ^2 u$ and the second one by $\rhat ^2 v$, and let us integrate over $[0,1]$. In this case, we obtain 
		\begin{multline}\label{eq2}
		\int_{0}^{1}\rhat ^2 [u_t u + v_t v ]-\int _{0}^{1} \rhat ^2 [(au_x) _x u + (av_x )_x v]\\
		= - \int_{0}^{1} \hat{\rho}^2[(b_{11} u + b_{12} v)u+ (b_{21} u + b_{22} v)v]  + \int _{0}^{1} \rhat ^2 h \chi _{\omega } u+ \int _{0}^{1} \rhat ^2 [g_1 u + g_2 v].
		\end{multline}
		Clearly, the terms in the right hand side of \eqref{eq2} can be estimated as follows:
		\begin{align*}
		& \displaystyle \left|  \int_{0}^{1} \hat{\rho}^2[(b_{11} u + b_{12} v)u+ (b_{21} u + b_{22} v)v]  \right| \leq C \int_{0}^{1} \rhat ^2 (|u|^2 + |v|^2)  ,\\
		&\displaystyle \int _{0}^{1} \rhat ^2 h \chi _{\omega } u   = \int _{\omega} (\rho _{0} u)(\rho _{\ast } h)   
		\leq \frac{1}{2} \int _{\omega} \rho _{0}^{2} |u|^2   
		+\frac{1}{2} \int _{\omega} \rho _{\ast }^{2} |h|^2  
		\end{align*}
		and
		\begin{multline*}
		\displaystyle \int _{0}^{1} \rhat ^2 [g_1 u + g_2 v] 
		\leq \frac{1}{2} \int _{0}^{1} \rhat ^2 (|g_1 |^2 + |u|^2 +|g_2 |^2 + |v|^2 ) \\ \leq C\left( \int_{0}^{1} \rho_{0}^{2} (|g_1 |^2 + |g_2 |^2)  + \int_{0}^{1} \rho_{0}^{2} (|u|^2 + |v|^2) \right) .
		\end{multline*}
		
		Now, let us deal with the left hand side of \eqref{eq2}. Notice that
		\begin{align*}
		\displaystyle \int _{0}^{1} \rhat ^2 [u_t u + v_t v ]  
		&=\frac{1}{2} \frac{d}{dt} \int_{0}^{1} \rhat ^2 (|u|^2 + |v|^2) 
		-\int_{0}^{1} \rhat (\rhat )_t  (|u|^2 + |v|^2 )  \\
		&:= \frac{1}{2} \frac{d}{dt} \int_{0}^{1} \rhat ^2 (|u|^2 + |v|^2) 
		-\mathcal{I}
		\end{align*}
		and	
		\begin{multline*}
		-\int _{0}^{1} \rhat ^2 [(au_x) _x u + (av_x )_x v] 
		=\int_{0}^{1} [(\rhat ^2 u)_x (au_x ) + (\rhat v )_x (av_x )]  \\
		= 2\int_{0}^{1} \rhat \rhat _x a (uu_x + vv_x )  + \int_{0}^{1} \rhat ^2 a (u_{x}^{2} + v_{x}^{2 })   := \mathcal J + \int_{0}^{1} \rhat ^2 a (u_{x}^{2} + v_{x}^{2 }).
		\end{multline*}
		
		Summing up, we have just checked that
		\begin{multline}\label{eq3}
		\frac{1}{2} \frac{d}{dt} \int_{0}^{1} \rhat ^2 (|u|^2 + |v|^2) 
		+ \int_{0}^{1} \rhat ^2 a (|u_{x}|^{2} + |v_{x}|^{2 })\\ \leq C\left(   \int_{0}^{1} \rho_{0}^{2} (|u|^2 + |v|^2)+\int _{\omega} \rho _{\ast }^{2} |h|^2+\int_{0}^{1} \rho_{0}^{2} (|g_1 |^2 + |g_2 |^2)  \right) +|\mathcal{I}|+|\mathcal J|
		\end{multline}

		Next, we will estimate $ \mathcal I$. Firstly, we put 
		\[A(t,x)=\tau(t)(e^{\lambda(|\psi|_\infty+\psi)}-e^{3\lambda|\psi|_\infty}):= \zeta (t,x) \bar{\eta } (x),\]
		where $\displaystyle \bar{\eta} (x):= \frac{e^{\lambda (|\psi |_{\infty} + \psi )} - e^{3\lambda |\psi |_{\infty }}}{\eta (x)} $ is a bounded function on $[0,1]$. Secondly, we observe that 
		\begin{align*}
		\displaystyle \rhat (\rhat )_t
		&=e^{-sA} \zeta ^{-5/2} \left[ -se^{-sA} A_t \zeta ^{-5/2} + e^{-sA} \left( \frac{-5}{2} \right)  \zeta ^{-7/2} \zeta _t \right]  \\
		&=-se^{-2sA} \bar{\eta } (x) \zeta ^{-5} \zeta _t - \frac{5}{2} e^{-2sA} \zeta ^{-6} \zeta _{t} \\
		&=e^{-2sA} \zeta ^{-2} \zeta _t \left[ -s\zeta ^{-3} \bar{\eta} (x) - \frac{5}{2} \zeta ^{-4} \right] \\
		&=\rho_{0}^{2} \zeta _t \left[ -s\zeta ^{-3} \bar{\eta} (x) - \frac{5}{2} \zeta ^{-4} \right] 
		\end{align*}
		and, for any $t\in [0,T)$, we have
		\begin{align*}
		\displaystyle |\rhat (\rhat )_t |
		&\leq C \rho_{0}^{2} \tau ^2 \left|s\zeta ^{-3} \bar{\eta } (x) - \frac{5}{2} \zeta ^{-4} \right| 
		= C \rho_{0}^{2} \zeta ^2 \frac{\left|s\zeta ^{-3} \bar{\eta } (x) - \frac{5}{2} \zeta ^{-4} \right| }{\eta ^2 } \\
		&= C \rho_{0}^{2} \frac{\left|s\zeta ^{-1} \bar{\eta } (x) - \frac{5}{2} \zeta ^{-2} \right| }{\eta ^2 } 
		\leq C\rho_{0}^{2} .
		\end{align*}	
		From this, we obtain 
		\[|\mathcal I| \leq C\int_{0}^{1} \rho_{0}^{2} (|u|^2 + |v|^2 ).\]
		
		Now, in order to deal with $\displaystyle \mathcal J $, we consider the estimate
		\begin{equation*}
		\displaystyle |\mathcal J|
		\leq \frac{1}{2} \int_{0}^{1} [\rhat ^2 a (|u_x |^2 + |v_x |^2 ) + \rhat _{x}^{2} a(|u|^2 +|v|^2 )] := \frac{1}{2} \int_{0}^{1} \rhat ^2 a (|u_x |^2 + |v_x |^2 )   + \tilde{\mathcal J} 
		\end{equation*} 
		and recall  that $A_x = \zeta _x = \lambda \zeta \psi _x $ in $[0,T]\times [0,1]$. Hence, 
		\begin{align*}
		\displaystyle \rhat_{x}^{2} 
		&=\left( -e^{-sA} sA_x  \zeta ^{-5/2} -\frac{5}{2} e^{-sA} \zeta ^{-7/2} \zeta_{x} \right) ^2 
		\leq C \left( e^{-2sA} s^2 A_{x}^{2}  \zeta ^{-5} +\frac{25}{4} e^{-2sA} \zeta ^{-7} \zeta_{x}^{2} \right) \\
		&=Ce^{-2sA} \zeta ^{-2} \psi _{x}^{2} \left( s^2 \lambda ^2 \zeta ^{-1} + \frac{25}{4} \lambda ^2 \zeta ^{-3} \right) 
		=C\rho _{0}^{2} \psi _{x}^{2} \left( s^2 \lambda ^2 \zeta ^{-1} + \frac{25}{4} \lambda ^2 \zeta ^{-3} \right) \\
		&\leq C\rho_{0}^{2}
		\end{align*}
		and we get $\displaystyle \tilde{\mathcal J} \leq C\int_{0}^{1} \rho_{0}^{2} a (|u|^2 + |v|^2 ) $.
		
		Recalling inequality \eqref{eq3}, we  conclude that
		\begin{multline*}
		\displaystyle  \frac{1}{2} \frac{d}{dt} \int_{0}^{1} \rhat ^{2} (|u|^2 + |v|^2 )   + \frac{1}{2} \int_{0}^{1} \rhat ^2 a (|u_x |^2 + |v_x |^2 )   \\ 
		\leq C\left( \int_{0}^{1} \rho_{0}^{2} (|u|^2 + |v|^2 )    + \int_{\omega} \rho_{\ast }^2 |h|^2    + \int_{0}^{1} \rho_{0}^{2} (|g_1 |^2 + |g_2 |^2 )  \right) 
		\end{multline*}
		and, integrating in time, we obtain the desired result. 
	\end{proof}
	
	\begin{prop} \label{ad2}
		Assume the hyphotesis of Theorem $\ref{linearcontrol}$ and suppose that $h$ and $(u,v)$ satisfy ($\ref{eq25}$). Then 
		\begin{multline*}
		\displaystyle \intq \rho_{\ast}^{2} (|u_t |^2 + |v_t |^2 + |(au_x )_x |^2 +|(av_x )_x |^2 )
		\leq C\intq \rho_{0}^{2} (|u|^2 + |v|^2 ) dt \\
		+C \left( \intw \rho_{\ast }^{2} |h|^2 + \intq \rho_{0}^{2} (|g_1 |^2 + |g_2 |^2 ) + \| u_0 \|_{H_{a}^{1}}^{2} + \| v_0 \|_{H_{a}^{1}}^{2} \right) .  
		\end{multline*}   
	\end{prop}
	
	\begin{proof}
		Multiplying the first equation in (\ref{pb-lin}) by $\rho_{\ast}^{2} u_t$ and the second one by $\rho_{\ast}^{2} v_t$, ,   we take
		\begin{multline*}
		\int_{0}^{1} \rho_{\ast}^{2} (|u_t|^2 + |v_t|^2 ) 
		-\int_{0}^{1} \rho_{\ast }^{2} [(au_x )_x u_t +(av_x )_x v_t ]  \\
		=-\int_{0}^{1} \rho_{\ast }^{2} [(b_{11} u + b_{12} v)u_t + (b_{21} u + b_{22} v )v_t ]+\int_{0}^{1} \rho_{\ast}^{2} h \chi_{\omega } u_t  +\int_{0}^{1} \rho_{\ast}^{2} (g_1 u_t + g_2 v_t ) .
		\end{multline*}
		Notice that, 
		\begin{align}\label{twice}
		\displaystyle -\int_{0}^{1} \rho_{\ast }^{2} [(au_x )_x u_t +(av_x )_x v_t ] 
		=&\frac{1}{2} \frac{d}{dt} \int_{0}^{1} \rho_{\ast}^{2} a (|u_x|^2 + |v_x|^2)\\ & -\frac{1}{2} \int_{0}^{1} (\rho_{\ast}^{2} )_t a (|u_x|^2 + |v_x|^2) 
		+\int_{0}^{1} (\rho_{\ast}^{2} )_x a(u_t u_x + v_t v_x ) \nonumber \\
		=:& \frac{1}{2} \frac{d}{dt} \int_{0}^{1} \rho_{\ast}^{2} a (|u_x|^2 + |v_x|^2) -\mathcal K , 
		\end{align}
		Proceeding as in the proof of Proposition \ref{ad1}, 
		\begin{multline}\label{eq4}
		\int_{0}^{1} \rho_{\ast}^{2} (|u_t|^2 + |v_t|^2 ) 
		+\frac{1}{2} \frac{d}{dt} \int_{0}^{1} \rho_{\ast}^{2} a (|u_x|^2 + |v_x|^2 )  \\
		\leq C\left(\int_{0}^{1} \rho_{\ast}^{2} (|u|^2 + |v|^2 )  
		+\int_{\omega} \rho_{\ast}^{2} |h|^2   
		+\int_{0}^{1} \rho_{\ast}^{2} (|g_1|^2 + |g_2|^2 ) \right)\\+\frac{3}{8} \int_{0}^{1} \rho_{\ast}^{2} (|u_t|^2 + |v_t|^2 )+|\mathcal{K}|.
		\end{multline}
		
		Using Young's inequality with $\varepsilon$, we have 
		\begin{align}\label{ca}
		\displaystyle \mathcal K
		&=\frac{1}{2} \int_{0}^{1} (\rho_{\ast}^{2} )_t a (|u_x|^2 + |v_x|^2) 
		-\int_{0}^{1} [((\rho_{\ast}^{2} )_x \rho_{\ast}^{-1} a u_x ) (\rho_{\ast} u_t ) + ((\rho_{\ast}^{2} )_x \rho_{\ast}^{-1} a v_x ) (\rho_{\ast} v_t )] \nonumber \\
		&\leq C \int_{0}^{1} [ |(\rho_{\ast}^{2} )_t | + |(\rho_{\ast}^{2})_x |^{2} \rho_{\ast}^{-2} ] a (|u_x|^2 + |v_x|^2) 
		+\frac{1}{8} \int_{0}^{1} \rho_{\ast}^{2} (|u_t|^2 + |v_t|^2 ).
		\end{align}
		Since $|\zeta_t |\leq C \zeta ^2$ and $\zeta_{x} = A_x = \lambda \zeta \psi _{x}$ in $[0,T)\times [0,1]$, we have
		\[\displaystyle |(\rho_{\ast}^{2} )_t|
		=\rhat ^2 |2sA_t \zeta ^{-3} + 8\zeta ^{-4} \zeta_t |
		\leq C\rhat^2 (2s\zeta^{-1} + 8\zeta^{-2} )\]
		and
		\begin{align*}
		\displaystyle |(\rho_{\ast}^{2} )_x |^2 \rho_{\ast }^{-2}
		&\leq C(4s^2 e^{-4sA} A_{x}^{2} \zeta ^{-16} + 64 e^{-4sA} \zeta ^{-18} \zeta_{x}^{2})\rho_{\ast }^{-2} \\
		&=C\rhat^2 (4s^2 A_{x}^{2} \zeta ^{-3} + 64 \zeta ^{-5} \zeta_{x}^{2}) \\
		&=C\rhat^2 (4s^2 \lambda ^2 \psi_{x}^{2} \zeta ^{-1} + 64 \zeta ^{-3}\lambda^2 \psi_{x}^{2} ),
		\end{align*}
		which imply
		\[\displaystyle \mathcal K \leq  C \int_{0}^{1} \rhat^2 a (|u_x|^2 + |v_x|^2) 
		+\frac{1}{8} \int_{0}^{1} \rho_{\ast}^{2} (|u_t|^2 + |v_t|^2 ).\]
		%
		Thus, from $\eqref{eq4}$,  applying Proposition $\ref{ad1}$ and using $\rho_{\ast } \leq C \rho_{0}$, we get
		\begin{multline}\label{part1}
		\intq \rho_{\ast}^{2} (|u_t |^2 + |v_t |^2 )
		\leq C\left(\intq \rho_{0}^{2} (|u|^2 + |v|^2 )  + \intw \rho_{\ast }^{2} |h|^2\right.\\ + \left.\intq \rho_{0}^{2} (|g_1 |^2 + |g_2 |^2 ) + \| u_0 \|_{H_{a}^{1}}^{2} + \| v_0 \|_{H_{a}^{1}}^{2} \right) .  
		\end{multline} 
		
		In order to conclude the proof, it remains to estimate  $\intq \rho_{\ast}^{2} (|(au_x )_x |^2 +|(av_x )_x |^2 )$. In fact, it is enough to multiply  the first equation  in ($\ref{pb-lin}$) by $-\rho_{\ast }^{2} (au_x)_x$ and the second one by  $-\rho_{\ast }^{2} (av_x)_x$, and proceed as in the first part of this proof.
	\end{proof}
	
	\section{Main Result } \label{sec-main}
	
	In this section, our goal is to prove Theorem \ref{main}.	Let us define the functions spaces
	\begin{multline*}
	\E:= \bigg\{  (u,v,h)\in\mathbf{ \left[L^2(\dom)\right]^2}\times L^2(\domw):\\
	u(t,\cdot),v(t,\cdot)  \text{ are absolutely continuous in } [0,1], \text{ a.e. in } [0,T], \\
	u_t, u_x, (au_x)_x, \rho_{\ast} h\in L^2 ((0,T)\times (0,1)),\ \   
	v_t, v_x, (av_x)_x \in L^2 ((0,T)\times (0,1)) ,\\
	\rho_{0} u, \rho_{0} [u_t - (au_x)_x -h\chi_{\omega}],
	\rho_{0} v, \rho_{0} [v_t - (av_x)_x ] \in L^2 ((0,T)\times (0,1)),\\ 
	u(t,1)\equiv v(t,1)\equiv u(t,0)\equiv v(t,0)\equiv 0 \text{ a.e in } [0,T], \emph{ and } 
	u(0,\cdot), v(0,\cdot)\in H_{a}^{1}\bigg\},
	\end{multline*}
	and $F:=G \times G \times H_{a}^{1} \times H_{a}^{1}$, where
	\[\displaystyle G := \left\lbrace  g\in L^2((0,T)\times (0,1)): \rho_{0} g \in L^2 ((0,T)\times (0,1))   \right\rbrace .\]
	
	We also consider the Hilbertian norm
	\begin{multline*}
	\n{(u,v,h)}{_E}^2:=
	\intq \rho_{0}^{2} (|u|^2 +|v|^2)
	+\intw \rho_{\ast}^{2} |h|^2 \\
	\hspace{0.5cm} +\intq \rho_{0}^{2} |u_t - (au_x)_x -h\chi_{\omega}|^2
	+\intq \rho_{0}^{2} |v_t - (av_x)_x |^2 \\
	\hspace{0.5cm}+\|u(0,\cdot)\|_{H_{a}^{1}}^{2}
	+\|v(0,\cdot)\|_{H_{a}^{1}}^{2} .
	\end{multline*}
	The proof that $E$ is a Hilbert space is given in Appendix \ref{hilbert}.

	Now, we set the mapping $H: E\to F$, given by
	\[H (u,v,h) = (H_{1} (u,v,h), H_{2} (u,v,h), u(\cdot ,0), v(\cdot ,0) ),\]
	where 
	\[\displaystyle H_{1} (u,v,h) = u_t-\left(\mu_1 \left(x,\int_0^1u \   \right)u_x \right)_x+f_1(t,x,u,v) -h\chi_\omega\]
	and
	\[\displaystyle H_{2} (u,v,h) = v_t-\left(\mu_2 \left(x,\int_0^1v \   \right)v_x \right)_x+f_2(t,x,u,v) . 
	\]
	Applying \textit{Lyusternik's Inverse Mapping Theorem}, see \cite{alekseev1987optimal}, we will prove that $H$ has a right inverse mapping defined in a small ball contained in  $F$.
	Due to the choice of the spaces $E$ and $F$, the existence of that inverse mapping will imply the local null controllability of ($\ref{pb1}$). Before doing it, we will establish some results which will guarantee that $H$ satisfies the hypotheses of Lyusternik's Theorem.
	
	\begin{lemma}\label{supremo}
		Define $\displaystyle \beta (x) =e^{\lambda (|\psi|_{\infty} +\psi)}-e^{3\lambda |\psi|_{\infty}} $ and $\bar{\beta} =\displaystyle \max_{x\in[0,1]} \beta(x)$. There exists $s>0$ such that, if $s\bar{\beta} <M<0$, then
		\[
		\displaystyle \sup_{t\in[0,T]} \left\lbrace e^{\frac{-2M}{m(t)}} \left[ \left( \int_{0}^{1} u \right)^2 +\left( \int_{0}^{1} v \right)^2 \right] \right\rbrace 
		\leq C\|(u,v,h)\|_{_E}^{2},\]
		for all $(u,v,h)\in E$.
	\end{lemma}
	
	\begin{proof}
		In fact, for each $(u,v,h)\in E$, consider $q_1 : [0,T]\longrightarrow \R$ and $q_2 : [0,T]\longrightarrow \R$, given by
		\[
		\displaystyle q_1 (t) = e^{\frac{-M}{m(t)}} \int_{0}^{1} u(t,x) 
		\emph{  and  }
		q_2 (t) = e^{\frac{-M}{m(t)}} \int_{0}^{1} v(t,x) .
		\]
		Taking $k>0$, we quickly get $\displaystyle e^{\frac{-k}{m(t)}} \leq 8! [m(t)]^8/k^8$,
		for any $t\in [0,T]$. Since $A=\tau \beta$, taking $s>0$ such that $2s (\bar{\beta} - \beta) >k$, we have
		\[\displaystyle -\frac{2M}{m(t)} +2sA =-\frac{2M}{m(t)} +\frac{2s\beta}{m(t)} \leq \frac{-2s(\bar{\beta} -\beta )}{m(t)} <\frac{-k}{m(t)}\]
		and
		\[
		\displaystyle e^{\frac{-2M}{m(t)}} < e^{-2sA} e^{\frac{-k}{m(t)}} \leq Ce^{-2sA} \tau^{-8} \leq C\rho_{\ast}^{2} ,
		\]
		for any $t\in [0,T]$. From this point, we may argue as in \cite{jrl2016} (see Lemma 4.4, on page 533), in order to check that $q_1,q_2\in H^1(0,T)$ and 
		\[
		\|q_1\|_{H^1(0,T)} +\|q_2\|_{H^1(0,T)} 
		\leq C\|(u,v,h)\|_{E}^{2}.
		\]
		Therefore, the desired result is a consequence of the continuous embedding $ H^1(0,T)\hookrightarrow C(0,T)$.
	\end{proof} 
	
	As a consequence of Lemma $\ref{supremo}$, we deduce the useful result below:
	
	\begin{cor}\label{conseq}
		There exists $C>0$ such that
		
		\begin{equation*}
		\intq \rho_{0}^{2} \left( \int_{0}^{1} \bar{u} \right) ^2 |(au_x)_x|^2
		+\intq \rho_{0}^{2} \left( \int_{0}^{1} \bar{v} \right) ^2 |(av_x)_x|^2 \leq  C\|(u,v,h)\|_{_E}^{2}  \|(\bar{u},\bar{v},\bar{h})\|_{_E}^{2},
		\end{equation*}
		for any $(u,v,h),(\bar{u},\bar{v},\bar{h})\in E$.
	\end{cor}
	
	\begin{proof}
		Take $(u,v,h),(\bar{u},\bar{v},\bar{h})\in E$ and let $M<0$ be the constant mentioned in Lemma $\ref{supremo}$. Since $\rho_{0}^{2} \rho_{\ast}^{-2} = \zeta^6 $ and $\tau^6 \leq \frac{45}{4M^6} e^{\frac{-2M}{m}}$, applying Proposition $\ref{ad2}$ and Lemma $\ref{supremo}$, we obtain 
		
		\begin{align*}
		\displaystyle \intq \rho_{0}^{2} \left( \int_{0}^{1} \bar{u} \right) ^2 |(au_x)_x|^2
		&\leq \sup_{t\in[0,T]} \left\lbrace e^{\frac{-2M}{m}} \left(  \int_{0}^{1} \bar{u} \right) ^2 \right\rbrace 
		\intq \rho_{\ast}^{2} |(au_x)_x|^2 \\
		&\leq C\|(u,v,h)\|_{E}^{2}  \|(\bar{u},\bar{v},\bar{h})\|_{E}^{2} .
		\end{align*}
		Analogously, a similar estimate also holds to $\intq \rho_{0}^{2} \left( \int_{0}^{1} \bar{v} \right) ^2 |(av_x)_x|^2$.
	\end{proof}
	
	\begin{prop}\label{2parts}
		The mapping $H: E\to F$ has the following properties:
		\begin{itemize}
			\item[(a)] $H$ is well defined;
			\item[(b)] For each $(u,v,h)\in E$, let us define $\bar{f}_{1}^{i} = D_i f_{1}(t,x,u,v)$ and $\bar{f}_{2}^{i} = D_i f_{2}(t,x,u,v)$, with $i=3,4$. Then, the linear mapping $T:E\longrightarrow G$ and $S:E\longrightarrow G$, given by
			\begin{align*}
			\displaystyle T(\bar{u},\bar{v},\bar{h}) 
			&= \bar{u}_t 
			-\ell_{1}^{\prime} \left( \int_{0}^{1} u\right)
			\left( \int_{0}^{1} \bar{u}\right)(au_x)_x 
			-\ell_1 \left(\int_{0}^{1} u\right) (a\bar{u}_x)_x \\
			&\hspace{1cm}+\bar{f}_{1}^{3} \bar{u}
			+ \bar{f}_{1}^{4} \bar{v}
			-\bar{h}\chi_{\omega},
			\end{align*}
			and
			\begin{align*}
			\displaystyle S(\bar{u},\bar{v},\bar{h}) 
			&= \bar{v}_t 
			-\ell_{2}^{\prime} \left( \int_{0}^{1} v\right)
			\left( \int_{0}^{1} \bar{v}\right)(av_x)_x 
			-\ell_2 \left(\int_{0}^{1} v\right) (a\bar{v}_x)_x \\
			&\hspace{1cm}+\bar{f}_{2}^{3} \bar{u}
			+ \bar{f}_{2}^{4} \bar{v},
			\end{align*}
			are the Gateaux derivative of $H_1$ and $H_2$ at $(u,v,h)\in E$, respectively.
		\end{itemize}
		
		\begin{proof} \noindent	
			\begin{itemize}
				\item[(a)] For each $(u,v,h)\in E$, we must check that $H(u,v,h)\in F$. Of course, $H_3 (u,v,h)=u(\cdot,0)\in H_{a}^{1} (0,1)$ and $H_4 (u,v,h)=v(\cdot,0)\in H_{a}^{1} (0,1)$. Besides, recalling the Assumptions $\eqref{hyp_a}$ and $\eqref{hyp_f}$, and using Corollary $\ref{conseq}$, we take
				\begin{align*}
				\displaystyle &\intq \rho_{0}^{2} |H_1(u,v,h)|^2
				=\intq \rho_{0}^{2} \left| u_t-\ell_1 \left(\int_0^1 u \right)(au_x)_x+f_1(t,x,u,v) -h\chi_\omega\right| ^2 \\
				&\hspace{0.5cm} \leq 3 \intq \rho_{0}^{2} |u_t - (au_x)_x -h\chi_{\omega}|^2
				+3\intq \rho_{0}^{2} \left| \ell_1 \left( \int_{0}^{1} u\right) -\ell_1 (0) \right| ^2 |(au_x)_x|^2 \\
				&\hspace{1cm}+ 3 \intq \rho_{0}^{2} |f_1 (t,x,u,v) - f_1 (t,x,0,0)|^2 \\
				&\hspace{0.5cm} \leq 3 \intq \rho_{0}^{2} |u_t - (au_x)_x -h\chi_{\omega}|^2
				+C\intq \rho_{0}^{2} \left( \int_{0}^{1} u\right) ^2 |(au_x)_x|^2 \\
				&\hspace{1cm}+C \intq \rho_{0}^{2} (|u|^2+|v|^2) \\
				&\hspace{0.5cm}\leq C(\|(u,v,h)\|_{E}^{2} +\|(u,v,h)\|_{E}^{4} ) 
				\end{align*}
				Therefore, $H_1 (u,v,h)\in G$ and, in a similar way, we also have $H_2 (u,v,h)\in G$.
				
				\item[(b)] In this part, fix $(u,v,h)\in E$. Thus, for any $(\bar{u},\bar{v},\bar{h})\in E$ and $\lambda \neq 0$, we take
				\begin{align*}
				\displaystyle &\frac{1}{\lambda} [H_1(u+\lambda \bar{u},v+\lambda \bar{v},h+\lambda \bar{h}) - H_1(u,v,h)] -T(\bar{u},\bar{v},\bar{h}) \\
				&\hspace{0.5cm} =-\left[ \frac{1}{\lambda} \left( \ell_1 \left( \int_{0}^{1} (u+\lambda \bar{u} )\right)  - \ell_1 \left( \int_{0}^{1} u\right)  \right) 
				-\ell_{1}^{\prime} \left( \int_{0}^{1} u\right) \left( \int_{0}^{1} \bar{u} \right) \right] (au_x)_x \\
				&\hspace{1cm} -\left[ \ell_1 \left( \int_{0}^{1} (u+\lambda \bar{u} )\right)  - \ell_1 \left( \int_{0}^{1} u\right) \right] (a\bar{u} _x)_x \\
				&\hspace{1cm} +\left[ \left( \frac{f_1(t,x,u+\lambda \bar{u},v+\lambda \bar{v})-f_1 (t,x,u,v)}{\lambda} \right) - (\bar{f}_{1}^{3} \bar{u} + \bar{f}_{1}^{4} \bar{v} )  \right] \\
				&\hspace{0.5cm}:=A_{\lambda}+B_{\lambda}+C_{\lambda}.
				\end{align*}
				We will see that $A_{\lambda}, B_{\lambda} \emph{ and } C_{\lambda}$ converge to zero in $G$, as $\lambda \to 0$. Indeed, taking into account  $\eqref{hyp_a}$ and Mean Value Theorem, for each $\lambda \neq 0$, there exists
				$u_{\lambda}=u_{\lambda}(t)$ such that $u_{\lambda} \to \int_{0}^{1} u$ for any $t\in[0,T]$,
				\[
				\displaystyle \intq \rho_{0}^{2} |A_{\lambda}|^2 = \intq \rho_{0}^{2}
				\left|  \ell_{1}^{\prime} (u_\lambda)  - \ell_{1}^{\prime} \left( \int_{0}^{1} u\right) \right|^{2}
				\left( \int_{0}^{1} \bar{u} \right)^2 |(au_x)_x|^2
				\to 0
				\]
				and
				\begin{align*}
				\displaystyle \intq \rho_{0}^{2} |B_{\lambda}|^2 
				&=\lambda^2 \intq \rho_{0}^{2} |\ell_{^1}^{\prime} (u_{\lambda})|^2 \left( \int_{0}^{1} \bar{u} \right)^2 |(au_x)_x|^2 \\ 
				&\leq C\lambda^2 \|(u,v,h)\|_{_E}^{2}  \|(\bar{u},\bar{v},\bar{h})\|_{_E}^{2} \to 0,
				\end{align*}
				as $\lambda\to 0$.
				On the other hand, for each $\lambda\neq 0$, we can apply again Mean Value and Lebesgue's Theorem, in order to obtain $w_{\lambda}=w_{\lambda} (t,x)$ satisfying: $w_{\lambda} \to (t,x,u,v)$, for any $(t,x)\in (0,T)\times (0,1)$, and
				\begin{align*}
				\displaystyle \intq \rho_{0}^{2} |C_{\lambda}|^2
				&=\intq \rho_{0}^{2} |[D_3 f_1 (w_{\lambda})\bar{u} + D_{4} f_1 (w_{\lambda})\bar{v}]-[\bar{f}_{1}^{3} \bar{u}+\bar{f}_{1}^{4} \bar{v}]|^2 \\
				&\hspace{-2cm}\leq \intq \rho_{0}^{2} |D_3 f_1 (w_{\lambda})-\bar{f}_{1}^{3}|^2 |\bar{u}|^2
				+\intq \rho_{0}^{2} |D_4 f_1 (w_{\lambda})-\bar{f}_{1}^{4}|^2 |\bar{v}|^2 
				\to 0,
				\end{align*}
				as $\lambda \to 0$. 
				
				As a consequence, $T$ is the Gateaux derivative of $H_1$ at $(u,v,h)\in E$. Likewise, $S$ is the Gateaux derivative of $H_2$ at $(u,v,h)\in E$.
			\end{itemize}
		\end{proof}
		
		\begin{prop}\label{c1}
			The mapping $H:E\to  F$ is continuously differentiable.
		\end{prop}
		
		\begin{proof}
			Clearly, $H_3, H_4 \in C^1 (E,H_{a}^{1})$. Now, take $(u,v,h)\in E$ and let $((u^n,v^n,h^n))_{n=1}^{\infty}$ be a sequence which converges to $(u,v,h)$ in $E$. For each $(\bar{u},\bar{v},\bar{h})\in \bar{B}_{1}(0)\subset E$, we have proved in Proposition $\ref{2parts}$ that
			\begin{multline*}
			\displaystyle H_{1}^{\prime} (u,v,h)(\bar{u},\bar{v},\bar{h}) 
			= \bar{u}_t 
			-\ell_{1}^{\prime} \left( \int_{0}^{1} u \right)
			\left( \int_{0}^{1} \bar{u}\right)(au_{x})_x 
			-\ell_1 \left(\int_{0}^{1} u \right) (a\bar{u}_x)_x \\
			+\bar{f}_{1}^{3} \bar{u}
			+ \bar{f}_{1}^{4} \bar{v}
			-\bar{h}\chi_{\omega}
			\end{multline*}
			and
			\begin{multline*}
			\displaystyle H_{1}^{\prime} (u^n,v^n,h^n)(\bar{u},\bar{v},\bar{h}) 
			= \bar{u}_t 
			-\ell_{1}^{\prime} \left( \int_{0}^{1} u^n \right)
			\left( \int_{0}^{1} \bar{u}\right)(au_{x}^{n})_x 
			-\ell_1 \left(\int_{0}^{1} u^n \right) (a\bar{u}_x)_x \\
			+D_3 f_1 (t,x,u^n,v^n) \bar{u}
			+ D_4 f_1 (t,x,u^n,v^n) \bar{v}
			-\bar{h}\chi_{\omega}.
			\end{multline*}
			Thus,
			\begin{align*}
			\displaystyle &(H^{\prime}_{1} (u^n,v^n,h^n) - H^{\prime}_{1} (u,v,h))(\bar{u},\bar{v},\bar{h})
			=-\ell_{1}^{\prime} \left( \int_{0}^{1} u^n\right) \left( \int_{0}^{1} \bar{u} \right) [a(u^n-u)_x]_x \\
			&\hspace{.5cm} -\left[ \ell_{1}^{\prime} \left( \int_{0}^{1} u^n\right)-\ell^{\prime}_{1} \left( \int_{0}^{1} u \right) \right] \left( \int_{0}^{1} \bar{u} \right)   (au_x)_x 
			-\left[ \ell_{1} \left( \int_{0}^{1} u^n\right)-\ell_{1} \left( \int_{0}^{1} u \right) \right] (a\bar{u}_x)_x \\
			&\hspace{.5cm}+[D_3 f_1 (t,x,u^n,v^n)-D_3 f_1 (t,x,u,v)]\bar{u}
			+[D_4 f_1 (t,x,u^n,v^n)-D_4 f_1 (t,x,u,v)]\bar{v} \\
			&\hspace{.5cm}:=X_{1}^{n}+X_{2}^{n}+X_{3}^{n}+X_{4}^{n}+X_{5}^{n}.
			\end{align*}
			From assumption $\ref{hyp_a}$ and Corollary $\ref{conseq}$, we get
			\begin{align*}
			\displaystyle \intq \rho_{0}^{2} |X_{1}^{n}|^2 & 
			\leq C\intq \rho_{0}^{2} \left( \int_{0}^{1} \bar{u} \right) ^2 |[a(u^n-u)_x]_x|^2
			\\
			& \leq C\|(u^n-u,v^n-v,h^n-h)\|_{E}^{2}
			\to 0
			\end{align*}
			and
			\begin{align*}
			\intq \rho_{0}^{2} |X_{3}^{n}|^2 
			& \leq C\intq \rho_{0}^{2} \left( \int_{0}^{1} (u^n-u)\right)^2 |(a\bar{u}_x)_x|^2\\
			& \leq C\|(u^n-u,v^n-v,h^n-h)\|_{E}^{2}
			\to 0,		
			\end{align*}
			as $n\to +\infty$.
			
			On the other hand, due to Lemma $\ref{supremo}$ and assumptions $\ref{hyp_a}$ and $\ref{hyp_f}$, we obtain
			\begin{align*}
			\intq \rho_{0}^{2} |X_{2}^{n}|^2 & \leq\sup_{t\in[0,T]} \left\lbrace e^{\frac{-2M}{m}} \left( \int_{0}^{1} \bar{u} \right)^2 \right\rbrace \intq \rho_{\ast}^{2}
			\left| \ell_{1}^{\prime} \left( \int_{0}^{1} u^n\right)-\ell^{\prime}_{1} \left( \int_{0}^{1} u \right) \right|^2 |(au_x)_x|^2 \\
			&\hspace{.5cm}\leq \intq \rho_{\ast}^{2}
			\left| \ell_{1}^{\prime} \left( \int_{0}^{1} u^n\right)-\ell^{\prime}_{1} \left( \int_{0}^{1} u \right) \right|^2 |(au_x)_x|^2 \to 0,
			\end{align*}
			and
			\begin{align*}
			\intq \rho_{0}^{2} |X_{4}^{n}|^2& 
			\leq C\left( \intq \rho_{0}^{2} |D_3 f_1 (t,x,u^n,v^n) - D_3 f_1 (t,x,u,v)|^2 |\bar{u}|^2 \right)^{\frac{1}{2}}  \left( \intq \rho_{0}^{2} |\bar{u}|^2 \right)^{\frac{1}{2}} \\
			&\leq C\left( \intq \rho_{0}^{2} |D_3 f_1 (t,x,u^n,v^n) - D_3 f_1 (t,x,u,v)|^2 |\bar{u}|^2 \right)^{\frac{1}{2}} \to 0,
			\end{align*}
			as $n\to +\infty$, where we have also applied Lebegue's Theorem. Clearly, $\intq \rho_{0}^{2} |X_{5}^{n}|^2$ is similar to $\intq \rho_{0}^{2} |X_{4}^{n}|^2$ and we conclude that $H^{\prime}_{1}$ is a continuous mapping. Analogously, this conclusion remains valid to $H^{\prime}_{2}$. In this case, the proof is complete.
		\end{proof}
	\end{prop}
	
	\begin{proof}[\textbf{Proof of Theorem $\ref{main}$}]
		From Propositions $\ref{2parts}$ and $\ref{c1}$, we already know that the $H\in C^1(E,F)$. We state that $H^{\prime} (0,0,0):E\longrightarrow F$ is onto. In fact, consider $b_{11} (t,x)=D_3 f_1 (t,x,0,0)$, $b_{12} (t,x)=D_4 f_1 (t,x,0,0)$, $b_{21} (t,x)=D_3 f_2 (t,x,0,0)$ and $b_{22} (t,x)=D_4 f_2 (t,x,0,0)$ in ($\ref{pb-lin}$). Thus, given $(g_1,g_2,u_0,v_0) \in F$, we apply Theorem $\ref{linearcontrol}$ in order to obtain $(u,v,h)$ which solves ($\ref{pb-lin}$) and satisfies the relations in ($\ref{eq25}$). As a result, $(u,v,h)\in E$ and $H^{\prime}(0,0,0)(u,v,h)=(g_1, g_2,u_0,v_0)$, as we were supposed to check.
		
		Hence, by \textit{Lyusternik's Inverse Mapping Theorem} (Theorem \ref{lyusternik}) , there exist $r>0$ and a mapping $\tilde{H} :B_{r} (0) \subset F \longrightarrow E$ such that
		\[
		\displaystyle H(\tilde{H} (y)) = y \emph{  for each  } y\in B_{r}(0)\subset F .
		\]
		In particular, if $(\bar{u} _0,\bar{v} _0)\in H_{a}^{1} \times H_{a}^{1}$ and $\| (\bar{u}_0,\bar{v}_0)\|_{H_{a}^{1} \times H_{a}^{1}} < r$, we conclude that $(\bar{u},\bar{v},\bar{h})=\tilde{H} (0,0,\bar{u}_0,\bar{v}_0)\in E$ solves  $H(\bar{u},\bar{v},\bar{h}) = (0,0,\bar{u}_0,\bar{v}_0)$. Since $\intq \rho_{0}^{2} (|\bar{u}|^2+|\bar{v}|^2) <+\infty$, we get $\bar{u}(T,x)=\bar{v}(T,x)=0$ for any $x\in[0,1]$, following the result. 
	\end{proof}
	
	\section{Some Additional Comments}\label{sec6}

	As a first comment, we note that, in assumption \ref{hyp_a}, we have taken a weak type of degeneracy and so that Dirichlet boundary conditions are required in \eqref{pb1}. However, if we had chosen strong type degeneracy, see \cite{alabau2006carleman}, \eqref{pb1} it would be treated with Neumann conditions. In this context, we believe that analogous results can be obtained.
	
	Another interesting question is concerned with global null controlability to \eqref{pb1}, which does not seem to be simple. Perhaps, this kind of result relies on a global inverse mapping theorem, see \cite{de1994global}, but much more refined estimates are necessary.

	Under some changes in the Lemma \ref{supremo} and following the arguments presented here, Theorem \ref{main} can also be obtained if we consider \eqref{pb1} with  the diffusion coefficients 
	\[\left(\mu_1 \left(x,u\right)u_x \right)_x \text{ and } \left(\mu_1 \left(x,v\right)v_x \right)_x.\]
	
	Other important topics arrise from our current research:
	
	\begin{itemize}
		\item It would be very nice to obtain Theorem \ref{main} without imposing $\mu_1$ and $\mu_2$ have separated variables. Nevertheless, it is still an open problem.

		\item In the  system $\eqref{pb1}$, we can replace each nonlinearity  $f_i(t,x,u,v)$ by
		$f_i (t,x,u,v,u_x,v_x)$, with $i\in \{1,2\}$, in order to analyse whether it is possible to prove results about null controllability.

		\item Previously, in \cite{jrl2016}, we have obtained a local null controlability result for degenarate parabolic equations with nonlocal tems, which implies, throughout standard arguments, a local null boundary controllability result. However, the same fact can not be directly deduced for systems with a reduced number of controls, see \cite{ammar2011recent}. In other words, the boundary controllability of  
		\begin{equation*}
		\begin{cases}
		u_t-\left(\mu_1 \left(x,\int_{0}^{1} u\right)u_x \right)_x+f_1(t,x,u,v) =0, & (t,x)\in \dom,\\
		v_t-\left(\mu_2 \left(x,\int_{0}^{1} v \right)v_x \right)_x+f_2(t,x,u,v)=0, & (t,x) \in \dom,\\
		u(t,1)=v(t,0)=v(t,1)=0,\ u(t,0)= h(t) \  &  t\in (0,T),\\
		u(0,x)=u_0(x) \ \ \mbox{ and }\ \ v(0,x)=v_0(x), & x\in  (0,1).
		\end{cases} 
		\end{equation*}
		is a very interesting unknown issue.
		
	\end{itemize}
	
	\appendix

	\section{Wellposedeness of (\ref{pb1})}
	
	In this section, we will apply Galerkin's method in order to obtain a unique solution to ($\ref{pb1}$). Precisely, let us consider the functions $a$, $\ell_1$, $\ell_2$, $f_{1}$ and $f_2$ as in the assumptions $\ref{hyp_a}$ and $\ref{hyp_f}$. Additionally, let us suppose that 
	\[
	0<\ell _0 \leq \ell_1 , \ell_2 \leq L_1,\]
	where $\ell_0$ and $L_0$ are two positive constants. We observe that there exist $C_0 >0$ and $L_2 >0$ such that 
	\[	\displaystyle |f_i (t,x,s_1,s_2)|\leq C_0 (|s_1|+|s_2|)\]
	and
	\[	\displaystyle |\ell_{i}^{\prime} (s)| \leq L_2,\]
	for any $(t,x,s_1,s_2)\in [0,T]\times [0,1]\times \R \times \R$ and $s\in \R$, with $i\in \{1,2\}$. Under this conditions, we will prove the following result:
	\begin{thm}
		Take $T>0$. If $F_1 , F_2 \in L^{2} ((0,T)\times (0,1))$ and $u_0 , v_0\in H_{a}^{1} (0,1)$, then there exists a unique weak solution of		
		\begin{equation}
		\begin{cases}
		u_t-\ell_1 \left( \int_{0}^{1} u\right) (au_x)_x  +f_1(t,x,u,v) =F_1, & (t,x)\in \dom,\\
		
		v_t-\ell_2 \left( \int_{0}^{1} v\right) (av_x)_x+f_2(t,x,u,v)=F_2, & (t,x) \in \dom,\\
		u(t,0)=u(t,1)=v(t,0)=v(t,1)=0, &  t\in (0,T),\\
		u(0,x)=u_0(x) \ \ \mbox{ and }\ \ v(0,x)=v_0(x), & x\in  (0,1).
		\end{cases} 
		\label{pbum} 
		\end{equation}	 
	\end{thm}
	
	\begin{proof}
		Let $(w_i)_{i=1}^{\infty}$ be an orthonormal basis of $H_{a}^{1} (0,1)$ such that
		\[	-(a(x) w_{ix})_x = \lambda _i w_i.\]
		Fix $m\in \mathbb N ^{\ast}$. Due to Carathodory's theorem, there exist absolutely continuous functions $g_{im} = g_{im} (t)$ and $h_{im} = h_{im} (t)$, with $i\in \{ 1,\cdots ,m\}$ and $t\in [0,T]$, such that the functions
		\[		t\in [0,T] \longmapsto u_{m} (t) = \sum_{i=1}^{m} g_{im} (t)w_i \in H_{a}^{1} (0,1)\]
		and
		\[		t\in [0,T] \longmapsto v_{m} (t) = \sum_{i=1}^{m} h_{im} (t)w_i \in H_{a}^{1} (0,1)\]
		satisfy
		\begin{equation}
		\begin{dcases}
		\begin{multlined}[b][0.8\textwidth]
		(u_{mt},w)-\ell_1 \left( \int_{0}^{1} u_m \right) \left( (au_{mx})_x ,w\right) \\  +(f_1(t,x,u_m,v_m),w) =(F_1 ,w),\ (t,x)\in \dom,
		\end{multlined}\\
		\begin{multlined}[b][0.8\textwidth]
		(v_{mt},\tilde{w})-\ell_2 \left( \int_{0}^{1} v_m \right) \left( (av_{mx})_x\right)_x,\tilde{w} ) \\
		+(f_2(t,x,u_m,v_m),\tilde{w} )=(F_2 ,\tilde{w} ),\   (t,x) \in \dom,
		\end{multlined}\\
		u_m (0)\to u_0 \emph{ in } H^{1}_{a} (0,1), \\  
		v_m (0)\to v_0 \emph{ in } H^{1}_{a} (0,1),
		\end{dcases} 
		\label{pbap} 
		\end{equation}
		for any $w,\tilde{w} \in [w_1 , \cdots ,w_m]$, where $(\cdot ,\cdot )$ denotes the inner product in $L^2 (0,1)$. Next,  our goal is to prove three energy estimates to $u_m$ and $v_m$. 
		
		Firstly, taking $w=u_m$ and $\tilde{w} = v_m$ in ($\ref{pbap}$), we obtain	
		\begin{align*}
		& \frac{1}{2} \frac{d}{dt} (\|u_m \|_{L^2 (0,1)}^{2} +\|v_m \|_{L^2 (0,1)}^{2} )
		+\ell_1 \left( \int_{0}^{1} u_m \right) \| \sqrt{a} u_{mx}\|_{L^2 (0,1)}^{2}  +\ell_2 \left( \int_{0}^{1} v_m \right) \| \sqrt{a} v_{mx}\|_{L^2 (0,1)}^{2} \\
		& \leq C_1 (\|u_m \|_{L^{2} (0,1)}^{2} + \|v_m \|_{L^{2} (0,1)}^{2}) +  C_2 (\|F_1 \|_{L^{2} (0,1)}^{2} + \| F_2 \|_{L^{2} (0,1)}^{2}).
		\end{align*}
		Hence, Gronwall's inequality yields
		\begin{align}\label{E1}
		\displaystyle &\max_{t\in[0,T]} \|u_m (t) \|_{L^2 (0,1)}^{2} +\max_{t\in[0,T]} \|v_m (t) \|_{L^2 (0,1)}^{2} +2L_0 \int_{0}^{T} (\| \sqrt{a} u_{mx}\|_{L^2 (0,1)}^{2} 
		+\| \sqrt{a} v_{mx}\|_{L^2 (0,1)}^{2} ) \nonumber \\
		&\hspace{0.5cm} \leq C\left( \|u_0 \|_{L^{2} (0,1)}^{2} + \|v_0 \|_{L^{2} (0,1)}^{2} + \int_{0}^{T} (\|F_1 \|_{L^{2} (0,1)}^{2} + \| F_2 \|_{L^{2} (0,1)}^{2}) \right) \nonumber \\
		&\hspace{0.5cm} =: \mathcal K.
		\end{align}		
		Now, let us consider $w=u_{mt}$ and $\tilde{w}=v_{mt}$. Thus,				
		\begin{align*}
		& \begin{multlined}[t][0.9\textwidth]
		\|u_{mt} (t) \|_{L^2 (0,1)}^{2} +  \|v_{mt} (t) \|_{L^2 (0,1)}^{2} \\
		+\frac{1}{2} \frac{d}{dt} 
		\left( \ell_1 \left( \int_{0}^{1} u_m \right) \| \sqrt{a} u_{mx}\|_{L^2 (0,1)}^{2} 
		+\ell_2 \left( \int_{0}^{1} v_m \right) \| \sqrt{a} v_{mx}\|_{L^2 (0,1)}^{2} \right) 
		\end{multlined}\\
		&\begin{multlined}[t][0.9\textwidth]
		\leq C_{0} (\| u_m\|_{L^{2} (0,1)} + \| v_m \|_{L^{2} (0,1)}) (\| u_{mt} \|_{L^2 (0,1)} + \| v_{mt} \|_{L^{2} (0,1)} ) \\
		+C(\| F_1 \|_{L^{2} (0,1)} + \| F_2 \|_{L^{2} (0,1)})(\| u_{mt} \|_{L^{2} (0,1)} + \| v_{mt} \|_{L^{2} (0,1)}) \\
		+\frac{1}{2} \left| \int_{0}^{1} u_{mt}\right| \ell_{1}^{\prime} \left( \int_{0}^{1} u_m \right)  \| \sqrt{a} u_{mx}\|_{L^{2} (0,1)}^{2}  
		+\frac{1}{2} \left| \int_{0}^{1} v_{mt}\right| \ell_{2}^{\prime} \left( \int_{0}^{1} v_m \right)  \| \sqrt{a} v_{mx}\|_{L^{2} (0,1)}^{2} 
		\end{multlined}
		\end{align*}
		\begin{align*}
		&\begin{multlined}[t][0.9\textwidth]
		\leq C_{0} \left(\| u_m\|_{L^{2} (0,1)} + \| v_m \|_{L^{2} (0,1)}) (\| u_{mt} \|_{L^{2} (0,1)} + \| v_{mt} \|_{L^{2} (0,1)} \right) \\
		+C \left(\| F_1 \|_{L^{2} (0,1)} + \| F_2 \|_{L^{2} (0,1)})(\| u_{mt} \|_{L^{2} (0,1)} + \| v_{mt} \|_{L^{2} (0,1)}\right) \\
		+L_{2} \left(\| u_{mt} \|_{L^2 (0,1)}   \| \sqrt{a} u_{mx}\|_{L^{2} (0,1)}^{2}  
		+\| v_{mt} \|_{L^2 (0,1)}   \| \sqrt{a} v_{mx}\|_{L^{2} (0,1)}^{2} \right) 
		\end{multlined} \\ 
		&\begin{multlined}[t][0.9\textwidth]
		\leq C(\varepsilon ) \left(\| u_m \|_{L^2 (0,1)}^{2} + \|v_m \|_{L^2 (0,1)}^{2} + \| F_1 \|_{L^2 (0,1)}^{2} + \| F_2 \|_{L^2 (0,1)}^{2}\right) \\
		+\varepsilon \left(\| u_{mt} \|_{L^2 (0,1)}^{2} + \| v_{mt} \|_{L^2 (0,1)}^{2}\right) \\
		+L_2 \left(\| \sqrt{a} u_{mx}\|_{L^2 (0,1)}^{2} + \| \sqrt{a} v_{mx} \|_{L^2 (0,1)}^{2}\right)^2 . 
		\end{multlined}
		\end{align*}
		
		Thus, integrating in $[0,T]$, we get		
		\begin{align} \label{E2}
		\displaystyle &\int_{0}^{t} (\| u_{mt}\|_{L^2 (0,1)}^{2} +\| v_{mt}\|_{L^2 (0,1)}^{2})
		+L_{0} (\| \sqrt{a} u_{mx}\|_{L^2 (0,1)}^{2} + \| \sqrt{a} v_{mx}\|_{L^2 (0,1)}^{2}) \nonumber \\
		&\hspace{0.5cm} \leq \int_{0}^{T} (\| u_m \|_{L^2 (0,1)}^{2} + \| v_m \|_{L^2 (0,1)}^{2} + \| F_1 \|_{L^2 (0,1)}^{2} + \| F_2 \|_{L^2 (0,1)}^{2}) \nonumber \\
		&\hspace{1cm} +L_1 (\| \sqrt{a} u_{0x}\|_{L^2 (0,1)}^{2} + \| \sqrt{a} v_{0x} \|_{L^2 (0,1)}^{2}) \nonumber \\
		&\hspace{1cm}+L_2 \int_{0}^{t} (\|\sqrt{a} u_{mx}\|_{L^2 (0,1)}^{2} + \| \sqrt{a} v_{mx}\|_{L^2 (0,1)}^{2})(\| \sqrt{a} u_{mx}\|_{L^2 (0,1)}^{2} + \| \sqrt{a} v_{mx}\|_{L^2 (0,1)}^{2}) \nonumber \\
		&\hspace{1cm} \leq (\mathcal K T +\| \sqrt{a} u_{0x}\|_{L^2 (0,1)}^{2} + \| \sqrt{a} v_{0x} \|_{L^2 (0,1)}^{2}) e^{\frac{L_2 }{L_0} \int_{0}^{T} (\| \sqrt{a} u_{mx}\|_{L^2 (0,1)}^{2} + \| \sqrt{a} v_{mx} \|_{L^2 (0,1)}^{2})} \nonumber \\
		&\hspace{1cm} \leq  (\mathcal K T +\| \sqrt{a} u_{0x}\|_{L^2 (0,1)}^{2} + \| \sqrt{a} v_{0x} \|_{L^2 (0,1)}^{2}) e^{\frac{L_2}{2L_0} \mathcal K} \nonumber \\
		&\hspace{1cm} =: \mathcal K _1 ,
		\end{align}
		
		for any $t\in [0,T]$, where we have used estimate ($\ref{E1}$) and applied Gronwall's inequality.
		
		Finally, we will prove the last estimate which is necessary to build a weak solution of ($\ref{pbum}$). 
		
		In fact, taking $w=-(au_{mx})_{x}$ and $\tilde{w} = -(av_{mx})_{x}$ in ($\ref{pbap}$), we have 
		
		\begin{align*}
		\displaystyle &\frac{1}{2} \frac{d}{dt} (\|\sqrt{a} u_{mx}\|_{L^{2} (0,1)}^{2} + \| \sqrt{a} v_{mx} \|_{L^{2} (0,1)}^{2})
		+L_{0} (\| (au_{mx})_{x}\|_{L^{2} (0,1)}^{2} + \| (av_{mx})_{x}\|_{L^{2} (0,1)}^{2}) \\
		&\hspace{0.5cm}\leq [C_0 (\| u_m \|_{L^{2} (0,1)} + \| v_m \|_{L^{2} (0,1)}) + \| F_1 \|_{L^{2} (0,1)} + \| F_2 \|_{L^{2} (0,1)}] \\
		&\hspace{1cm} \cdot (\| (au_{mx})_x \|_{L^{2} (0,1)}^{2} + \| (av_{mx})_x\|_{L^{2} (0,1)}^{2}) \\
		&\hspace{0.5cm}\leq  \tilde C (\varepsilon) [ \mathcal K + \| F_1 \|_{L^{2} (0,1)}^{2} + \| F_2 \|_{L^{2} (0,1)}^{2} ] \\
		&\hspace{1cm} +\varepsilon (\| (au_{mx})_x \|_{L^{2} (0,1)}^{2} + \| (av_{mx})_x\|_{L^{2} (0,1)}^{2}).
		\end{align*}
		
		As a conclusion,	
		\begin{align}\label{E3}
		&(\| \sqrt{a} u_{mx} \|_{L^{2} (0,1)}^{2} + \| \sqrt{a} v_{mx} \|_{L^{2} (0,1)}^{2})
		+L_0 \int_{0}^{T} [\| (au_{mx})_{x} \|_{L^{2} (0,1)}^{2} + \| (av_{mx} )_x \|_{L^{2} (0,1)}^{2}] \nonumber \\
		&\hspace{0.5cm}\leq  2\tilde C (\varepsilon) [ \mathcal K + \| F_1 \|_{L^{2} (0,1)}^{2} + \| F_2 \|_{L^{2} (0,1)}^{2} ] 
		+(\| \sqrt{a} u_{0x} \|_{L^{2} (0,1)}^{2} + \| \sqrt{a} v_{0x} \|_{L^{2} (0,1)}^{2}) \nonumber \\
		&\hspace{0.5cm}=: \mathcal K _2 .
		\end{align}
		
		Since $\mathcal K _1$, $\mathcal K _2$ and $\mathcal K _3$ do not depend on $m\in \mathbb N ^{\ast}$, the three estimates ($\ref{E1}$), ($\ref{E2}$) and ($\ref{E3}$) imply that the sequences $(u_m)_{m=1}^{\infty}$ and $(v_m)_{m=1}^{\infty}$ are bounded in
		$$
		L^2 (0,T;L^2(0,1))\cap H^{1} (0,T; H_{a}^{2} (0,1)).  
		$$ 
		Therefore, there exist subsequences $(u_{m_{j}})_{j=1}^{\infty}$ of $(u_m)_{m=1}^{\infty}$ and
		$(v_{m_{j}})_{j=1}^{\infty}$
		$(v_m)_{m=1}^{\infty}$, such that 
		$$
		u_{m_{j}} \rightarrow u \text{ and } v_{m_{j}} \rightarrow v, \text{ as } m\to \infty
		$$
		weakly in $L^2 (0,T;L^2(0,1))\cap H^{1} (0,T; H_{a}^{2} (0,1))$. By standard arguments, we can conclude that $(u,v)$ is a weak solution of ($\ref{pbum}$). 
	\end{proof}
	
	\begin{rem}
		Above, we have proved the well-posedness of \eqref{pb1}, by assuming that $\ell_i (r) \geq C>0$, for $i=1,2$. However, Theorem $\ref{main}$ has been obtained under the Assumptions \ref{hyp_a} and \ref{hyp_f}, where just $\ell_i (0) >0$ is required. This last hypotheses is sufficient to prove that $H\in C^1 (E;F)$ is well defined and that $ H'(0,0,0)\in L(E;F)$ is onto.
	\end{rem}

	\section{$\E$ is a Hilbert Space} \label{hilbert}
	
	In this section we will prove that $\E$ is a Hilbert space. Let us recall that
	\begin{multline*}
	\E:= \bigg\{  (u,v,h)\in\mathbf{ \left[L^2(\dom)\right]^2}\times L^2(\domw):\\
	u(t,\cdot),v(t,\cdot)  \text{ are absolutely continuous in } [0,1], \text{ a.e. in } [0,T], \\
	u_t, u_x, (au_x)_x, \rho_{\ast} h\in L^2 ((0,T)\times (0,1)),   
	v_t, v_x, (av_x)_x \in L^2 ((0,T)\times (0,1)) ,\\
	\rho_{0} u, \rho_{0} [u_t - (au_x)_x -h\chi_{\omega}],
	\rho_{0} v, \rho_{0} [v_t - (av_x)_x ] \in L^2 ((0,T)\times (0,1)),\\ 
	u(t,1)\equiv v(t,1)\equiv u(t,0)\equiv v(t,0)\equiv 0 \text{ a.e in } [0,T], \emph{ and } 
	u(0,\cdot), v(0,\cdot)\in H_{a}^{1}\bigg\},
	\end{multline*}
	with the following norm induced by an inner product
	\begin{multline*}
	\n{(u,v,h)}{_E}^2:=
	\intq \rho_{0}^{2} (|u|^2 +|v|^2)
	+\intw \rho_{\ast}^{2} |h|^2 \\
	\hspace{0.5cm} +\intq \rho_{0}^{2} |u_t - (au_x)_x -h\chi_{\omega}|^2
	+\intq \rho_{0}^{2} |v_t - (av_x)_x |^2 \\
	\hspace{0.5cm}+\|u(0,\cdot)\|_{H_{a}^{1}}^{2}
	+\|v(0,\cdot)\|_{H_{a}^{1}}^{2} .
	\end{multline*}

	\begin{proof} 
		Let $(u_n,v_n,h_n)_{n=1}^\infty$ be a Cauchy sequence in $\E$. In particular, 
\begin{equation}\label{conv1}
	\begin{cases}
(\rho_0u_n)_{n=1}^\infty,\ (\rho_0v_n)_{n=1}^\infty,\ (\rho_\ast h_n\chi_{\omega })_{n=1}^\infty,\ 
 (\rho_{0} [{u_n}_{_t} - (a{u_n}_{_x})_x -h_n\chi_{\omega}])_{n=1}^\infty,  \\
  (\rho_{0} [{v_n}_{_t} - (a{v_n}_{_x})_x ])_{n=1}^\infty, \text{ are Cauchy in } L^2(\dom).\\
  (u_n(0,\cdot))_{n=1}^\infty,  (v_n(0,\cdot))_{n=1}^\infty, \text{ are Cauchy in } H_a^1. 
 	\end{cases}
\end{equation}
And, since $\rho_0,\rho_\ast\geq C_T$, we also have that
\begin{equation}\label{conv2}
\begin{cases}
(u_n)_{n=1}^\infty,\ (v_n)_{n=1}^\infty,\ (h_n\chi_{\omega })_{n=1}^\infty, \\
( {u_n}_{_t} - (a{u_n}_{_x})_x -h_n\chi_{\omega})_{n=1}^\infty,\ ( {v_n}_{_t} - (a{v_n}_{_x})_x )_{n=1}^\infty, \text{ are Cauchy in } L^2(\dom)
\end{cases}
\end{equation}
In particular, there exists $h\in L^2(\domw)$ such that 
\begin{equation}\label{conv4}
h_n\to h \text{ in } L^2(\domw).
\end{equation}

Now, let us set
\begin{equation*}
\begin{cases}
g_{1,n}:= {u_n}_{_t} - (a{u_n}_{_x})_x +b_{11}u_n+b_{12}v_n-h_n\chi_{\omega },\\
g_{2,n}:= {u_n}_{_t} - (a{v_n}_{_x})_x +b_{21}u_n+b_{22}v_n,\\
u_{0,n}:=u_n(0,\cdot), \ v_{0,n}:=v_n(0,\cdot).
\end{cases}
\end{equation*}
In this case, we can see that $(u_n,v_n)$ is a weak solution to
	\begin{equation*}
\begin{cases}
{u_n}_{_t} - (a{u_n}_{_x})_x +b_{11}u_n+b_{12}v_n =h_n\chi_\omega+g_{1,n}, & (t,x) \mbox{ in }\dom,\\
{u_n}_{_t} - (a{v_n}_{_x})_x +b_{21}u_n+b_{22}v_n =g_{2,n}, & (t,x) \mbox{ in }\dom,\\
u_n(t,1)=u_n(t,0)=v_n(t,0)=v_n(t,1)=0, &  t \mbox{ in } (0,T),\\
u_n(0,x)=u_{0,n}(x) \ \ \mbox{ and }\ \ v_n(0,x)=v_{0,n}(x), & x\mbox{ in } (0,1),
\end{cases} 
\end{equation*}
where $g_{1,n}, h_n\chi_{\omega }$,  $g_{2,n} \in L^2(\dom)$ and $u_{0,n},v_{0,n}\in H_a^1$.
Hence, Proposition \ref{prop-WP-lin} gives us that 
\[u_n,v_n \in  H^1(0,T;L^2(0,1))\cap L^2(0,T;H_a^2)\cap C^0([0,T];H_a^1)\]
and satisfies the inequality \eqref{ineq1}. Therefore, from \eqref{conv2} together with the third term of \eqref{conv1},   we have that $(u_n)_{n=1}^\infty,\ (v_n)_{n=1}^\infty$ are Cauchy sequences in the Banach space $ H^1(0,T;L^2(0,1))\cap L^2(0,T;H_a^2)\cap C^0([0,T];H_a^1)$. As a consequence, there exist $u,v$ such that 
\[ u_n \to u \text{ and } v_n\to v \text{ in } H^1(0,T;L^2(0,1))\cap L^2(0,T;H_a^2)\cap C^0([0,T];H_a^1).\]

This convergence together with \eqref{conv4} guarantee that 
\[(u_n,v_n,h_n)\to (u,v,h) \text{ in } E. \]
	\end{proof}

	\section{Some properties of $J_n$} \label{Jn}
	
	In this section, we will prove that the functional $J_n$ defined in Theorem \ref{linearcontrol} is lower semi-continuous, strictly convex and coercive. For convenience, let us recall its definition: 
	\begin{align*}
	J_n(u,v,h)& =\frac{1}{2}\intq \rho_{0,n}^2\left(|u|^2+|v|^2\right)+\frac{1}{2}\intq \rho_{\ast,n}^2|h|^2\\
	&= \frac{1}{2}\left(\n{\rho_{0,n}u}{L^2}^2+\n{\rho_{0,n}v}{L^2}^2 +\n{\rho_{\ast,n}h}{L^2}^2\right)
	\end{align*}
	where  $(u,v,h)\in \mathbf{ \left[L^2(\dom)\right]^3}$.
	
	Since $J_n$ is a sum of squared norms, it is strictly convex. In order to prove the remaining properties, we will need the following lemma.
	
	\begin{lemma}\label{lemma-const}
		There exist  constants $C_{n,T}>0$, depending on $n$ and $T$, and $C_T>0$, depending only on $T$,  such that
		\[ 0<C_T\leq \rho_{0,n}\leq C_{n,T} \text{ in } [0,T]\times [0,1]. \]
		and
		\[ 0<C_T\leq  \rho_{\ast,n}\leq C_{n,T} \text{ in } [0,T]\times [0,1]. \]
	\end{lemma}
	\begin{proof}
		
		We will prove the estimates to $\rho_{0,n}$, those corresponding to $\rho_{\ast,n}$ are analogous. Firstly, we note that $\zeta$ is bounded from below by a positive constant depending only on $T$, which we will denote by $m_T$. 
		
		Secondly, note that we can rewrite $A$ as
		\[A=- \left(\frac{e^{3\lambda|\psi|_\infty}}{\eta}-1\right)\zeta:=-\beta(x)\zeta, \]
		where $\beta(x)=\frac{e^{3\lambda|\psi|_\infty}}{\eta}-1$ is bounded and strictly positive for $\lambda$ large enough, i.e., there exist constants $\beta_0,\beta_1>0$ such that
		\[0<\beta_0 \leq \beta(x)\leq \beta_1, \ \forall x\in [0,1]. \]
		Thus,		 
		\[ -\beta_1\zeta \leq A\leq -\beta_0\zeta<0.\]
		
		
		If  $t\in [0,T/2]$, we can see that $\zeta(t)$ is bounded from above and below by positive constants which depend only on $T$, that is,  there exist constants $m_T,M_T>0$, depending only on $T$,   such that
		\[0<m_T\leq \zeta \leq  M_T.\]
		
		Now, notice that
		\[\frac{T^4}{16(T^4+1)}\leq \frac{(T-t)^4}{(T-t)^4+\frac{1}{n}}\leq 1 \]
		and, since $A_n=A\frac{(T-t)^4}{(T-t)^4+\frac{1}{n}}<0$, we have that 
		\[0<s\beta_0m_T\frac{T^4}{16(T^4+1)}\leq -sA_n\leq s\beta_1 M_T \Rightarrow 1\leq e^{-sA_n}\leq  e^{s\beta_1 M_T} .\]
		Therefore, since $\rho_{0,n}=e^{-sA_n}\zeta^{-2}$, we get that
		\[0< \frac{1}{M_T^2} \leq \rho_{0,n}\leq e^{s\beta_1 M_T}\frac{1}{m_T^2}. \]
		
		If $t\in [T/2,T]$, we have that \[A_n=-\beta\zeta\frac{(T-t)^4}{(T-t)^4+\frac{1}{n}}=-\beta\frac{\eta}{t^4}\frac{1}{(T-t)^4+\frac{1}{n}}.\]
		And, as a consequence,
		\[\frac{\beta_0}{T^4(\frac{T^4}{16}+1)}\leq -A_n\leq \frac{\beta_1|\eta|_\infty 16n}{T^4}. \]
		
		Hence, 
		\[ e^{-sA_n}\zeta^{-2}\leq e^{\frac{s\beta_1|\eta|_\infty 16n}{T^4}}m_T^{-2}:=C_{n,T}.\]
		Therefore, since $e^x\geq \frac{x^3}{3!}$ for all $x>0$, we finally conclude that
		\[ e^{-sA_n}\zeta^{-2}\geq \frac{-s^3A_n^3}{3!}\zeta^{-2} =\frac{s^{3}\beta^
			2}{3!}(-A_n)\geq \frac{s^3\beta_0^3}{T^4(\frac{T^4}{16}+1)}:=C_T.  \]
	\end{proof}
	
	\begin{prop}
		$J_n$ is lower semi-continuous and coercive.
	\end{prop}
	
	\begin{proof}
		Firstly, note that for each $n\in \mathbb{N}^\ast$, the last lemma gives us that the norms in the definition of $J_n$ are equivalents to the norms in $L^2(\dom)$.
		
		Given a sequence   $(u_k,v_k,h_k)_{k=1}^\infty$ in  $\mathbf{ \left[L^2(\dom)\right]^3}$  such that
		\[ (u_k,v_k,h_k) \to (u,v,h)\ \text{ in } \mathbf{ \left[L^2(\dom)\right]^3},\]
		we have 
		\[ \rho_{0,n}u_k\to \rho_{0,n} u,\ \rho_{0,n}v_k\to \rho_{0,n} v  \text{ and } \rho_{\ast,n}h_k\to \rho_{\ast,n} h \text{ in } L^2(\dom), \text{ as } k\to +\infty,
		\]
		where $n$ is fixed.
		In particular, 
		\[\n{\rho_{0,n}u_k}{L^2}\to \n{\rho_{0,n}u}{L^2}, \ \n{\rho_{0,n}v_k}{L^2}\to \n{\rho_{0,n}v}{L^2} \text{ and } \n{\rho_{0,n}h_k}{L^2}\to \n{\rho_{0,n}h}{L^2}.
		\]
		As a consequence,
		\begin{align*}
		J_n(u,v,h)&=\frac{1}{2}\left(\n{\rho_{0,n}u}{L^2}^2+\n{\rho_{0,n}v}{L^2}^2+\n{\rho_{\ast,n}h}{L^2}^2\right)\\
		&=\lim\limits_{k\to \infty}\frac{1}{2}\left(\n{\rho_{0,n}u_k}{L^2}^2+\n{\rho_{0,n}v_k}{L^2}^2 +\n{\rho_{\ast,n}h_k}{L^2}^2\right)\\
		& =\lim\limits_{k\to \infty}J_n(u_k,v_k,h_k),
		\end{align*}
		which proves that $J_n$ is continuous and consequently lower semi-continuous.
		
		Analogously, given a sequence  $(u_k,v_k,h_k)_{k=1}^\infty \text{ in } \mathbf{ \left[L^2(\dom)\right]^3}$  such that
		\[ \n{(u_k,v_k,h_k)}{ \mathbf{ \left[L^2(\dom)\right]^3}} \to +\infty,\]
		we have that
		\begin{align*}
		J_n(u_k,v_k,h_k)&=\frac{1}{2}\left(\n{\rho_{0,n}u_k}{L^2}^2+\n{\rho_{0,n}v_k}{L^2}^2 +\n{\rho_{\ast,n}h_k}{L^2}^2\right)\\
		& \geq C_T\left(\n{u_k}{L^2}^2+\n{v_k}{L^2}^2 +\n{h_k}{L^2}^2\right)\to+\infty.
		\end{align*}
		Therefore, $J_n$ is coercive.
	\end{proof}

%
	
\end{document}